\documentclass[10pt]{siamltex}

\usepackage{lipsum}
\usepackage{amsfonts}
\usepackage{graphicx}
\usepackage{epstopdf}
\usepackage{algorithmic}
\usepackage{amsfonts}
\usepackage{xmpmulti}
\usepackage{amssymb}
\usepackage{amsmath}
\usepackage{subfigure}
\usepackage{color}
\usepackage{array}
\usepackage{mathrsfs}
\usepackage{bm}
\usepackage{multirow}
\usepackage{multirow,array,color,mathrsfs,xcolor}
\usepackage{amsmath,amssymb}
\usepackage{booktabs,upgreek,todonotes}
\newtheorem{Defi}{Definition}
\newtheorem{Theo}{Theorem}

\newtheorem{Rem}{Remark}
\newtheorem{Lem}{Lemma}

\newtheorem{exam}{Example}

\def \b{\beta}
\def \a{\alpha}
\def \ilr {\mathcal{I}_{p,a,b}}
\def \dlr {\mathcal{D}_{p,a,b}}
\def \ilp {{_{a}}I_x}
\def \irp {{_{x}}I_{b}}
\def \dlxa {{_{a}}D_x}
\def \drxa {{_{x}}D_b}
\def \ilxa {{_{a}}I_x}
\def \irxa {{_{x}}I_b}
\def \dlx {{_{-1}}D_x}

\def \ilx {{_{-1}}I_x}
\def \irx {{_{x}}I_1}
\def \cdlxa  {{_{a}^{C}}D_x}
\def \cdrxa {{_{x}^{C}}D_b}
\def \cdlx  {{_{-1}^{C}}D_x}
\def \cdrx {{_{x}^{C}}D_1}

\def \ilxinf {{_{-\infty}}I_x}
\def \irxinf {{_{x}}I_\infty}

\def \ix {\mathcal{I}_{p}}
\def \dx {\mathcal{D}_{p}}

\def \dxc {{^C}\mathcal{D}_{p}}

\def \j {P}

 \makeatletter
\def\ps@pprintTitle{%
 \let\@oddhead\@empty
 \let\@evenhead\@empty
 \def\@oddfoot{\centerline{\thepage}}%
 \let\@evenfoot\@oddfoot}
\makeatother
\usepackage[mathlines]{lineno}

\title{A spectral penalty method for two-sided fractional differential equations with general boundary conditions\footnotemark[1]}
\author{Nan Wang\footnotemark[2] \footnotemark[3] \footnotemark[4]
\and Zhiping Mao\footnotemark[3]
\and Chengming Huang\footnotemark[2] \footnotemark[4]
\and George Em Karniadakis\footnotemark[3]
}

\begin{document}

\graphicspath{{FDBCsmou/}{FDBCsmof/}{FNBCsmou/}{FNBCsmof/}{FNBCsmof/}{CapDBCsmou/}{CapDBCsmof/}{CapFNBCsmou/}{CapFNBCsmof/}{Timedifuss/}}
\markboth{}{}

\maketitle

\footnotetext[1]{This work was supported by the MURI/ARO on ``Fractional PDEs for Conservation Laws and Beyond: Theory, Numerics and Applications" (W911NF-15-1-0562) and  NSF of China (No. 11771163).}
\footnotetext[2]{School of Mathematics and Statistics, Huazhong University of Science and Technology, Wuhan 430074, China.}
\footnotetext[3]{Division of Applied Mathematics, Brown University, Providence, RI 02912, USA ({zhiping\_mao@brown.edu}, {nan\_wang@brown.edu}, {george\_karniadakis@brown.edu}).}
\footnotetext[4]{Hubei Key Laboratory of Engineering Modeling and Scientific Computing, Huazhong University of Science and Technology, Wuhan 430074, China ({chengming\_huang@hotmail.com}).}


\begin{abstract}
We consider spectral approximations to the conservative form of the two-sided Riemann-Liouville (R-L) and Caputo fractional differential equations (FDEs) with nonhomogeneous Dirichlet (fractional and classical, respectively) and Neumann (fractional) boundary conditions. In particular, we develop a spectral penalty method (SPM)  by using  the Jacobi poly-fractonomial approximation for the conservative R-L FDEs while using the polynomial approximation for the conservative Caputo FDEs. We establish the well-posedness of the corresponding weak problems and analyze sufficient conditions for the coercivity of the SPM for different types of fractional boundary value problems.
This analysis allows us to estimate the proper values of the penalty parameters at boundary points.
We present several numerical examples to verify the theory and demonstrate the high accuracy  of SPM, both for stationary and time dependent FDEs. Moreover, we compare the results against a Petrov-Galerkin spectral tau method (PGS-$\tau$, an extension of \cite{MaoKar18}) and demonstrate the superior accuracy of SPM for all cases considered.
\end{abstract}



\section{Introduction}\label{sec1}

Fractional differential equations (FDEs) have been used effectively to model complex physical processes governed by non-local interactions,
for example,  modeling contaminant transport in rivers \cite{deng2006parameter}, the spread of invasive species \cite{baeumer2007fractional} and the transport at the earth surface \cite{schumer2009fractional}.
In particular, the two-sided fractional diffusion equation is required in applications such as hydrology~\cite{benson2000application, chakraborty2009parameter, zhang2016bounded} and plasma turbulent transport~\cite{del-Castillo-Negrete2006}.
However, one of the open problems in applying FDEs to real-world applications is the proper specification and numerical implementation of boundary conditions (BCs), consistent with the type of fractional derivatives involved, i.e., of Riemann-Liouville (R-L) type or Caputo type ~\cite{podlubny1998fractional}.
The most popular BCs used are the classical (local) Dirichlet BCs, see \cite{metzler2000generalized, metzler2000random, benson2000application,MeerMM04,del-Castillo-Negrete2006} and references therein. However, the diffusion equation with Dirichlet BCs does not conserve mass~\cite{KellyPreprint}. The classical (local) Neumann BCs have also been employed by many researchers~\cite{Montefusco2013DCDS, Szekeres2015OM, Xie2017BVP}.
Due to the non-locality of the fractional operator, the local BCs may not be suitable depending on the type of fractional derivative, hence non-local/fractional BCs have been considered in some other works, for instance, see \cite{zhang2007impact, Lim2009PLB, Zhou2010SINUM, WangYangZhu2014SIAMNUM, WangYang2017FCAA, Ma.J2017JSC}. Moreover, by imposing the no-flux BCs, namely, homogeneous fractional Neumann boundary conditions, we can recover the mass conservation~\cite{Baeumer2018JCAM, Baeumer2018JDE, KellyPreprint}.
However, the numerical implementation of non-local BCs is not straightforward and requires special treatment in order to preserve the accuracy of the numerical method used, especially in high order methods such as spectral Galerkin methods.

In this work, we consider the following conservative two-sided FDEs with \emph{general BCs}:
\begin{equation}\label{Intro1.1}
  -\frac{d}{dx} \,\mathbb{D}_{x}^{\alpha-1} u(x)+c u(x)=f(x),\quad x\in \Lambda:= (-1,1),
\end{equation}
where  $1<\a<2$, $f(x)$ is a given function.
The function $\mathbb{D}_{x}^{\alpha-1} u(x)$ can be considered as a flux function of the fractional diffusion equation in the conservative form~\cite{paradisi2001fractional, schumer2001eulerian}
\begin{equation}\label{Twosidediffusion1}
  \frac{\partial}{\partial t}u(x,t)-\frac{\partial}{\partial x} \,\mathbb{D}_{x}^{\alpha-1} u(x,t)=0.
\end{equation}
For the operator $\mathbb{D}_{x}^{\alpha-1}$,
we consider two types of fractional derivative, namely, R-L and Caputo.
Consequently, we study the following two types of fractional problem with the consistent BCs:
\begin{itemize}
  \item Conservative R-L FDEs, i.e., $\mathbb{D}_{x}^{\alpha-1} = \dx^{\alpha-1}$,
  with the R-L \emph{fractional Dirichlet boundary conditions} (FDBCs)
  \begin{equation}
  \ix^{2-\a}u(-1)=g_1,\quad  \ix^{2-\a}u(1)=g_2 \label{IntroFDBC1.2}
\end{equation}
or the R-L \emph{fractional Neumann boundary conditions} (FNBCs)
\begin{equation}
 \dx^{\alpha-1} u(-1)=g_1,\quad \dx^{\alpha-1}u(1)=g_2, \label{IntroFNBC1.3}
\end{equation}
  \item Conservative Caputo FDEs, i.e., $\mathbb{D}_{x}^{\alpha-1} = \dxc^{\alpha-1}$,
  with the \emph{classical local Dirichlet BCs}
      \begin{equation}
  u(-1)=g_1,\quad u(1)=g_2,  \label{IntroDBC1.4}
\end{equation}
which recovers the case of homogeneous Dirichlet boundary problem considered in \cite{MaoKar18} if $g_1 = g_2 = 0$,
or the Caputo \emph{FNBCs}
\begin{equation}
 \dxc^{\alpha-1} u(-1)=g_1,\quad \dxc^{\alpha-1}u(1)=g_2. \label{IntroFNBC1.5}
\end{equation}
\end{itemize}
The definitions of the fractional operators $\ix^s, \dx^s, \dxc^s, s>0$ can be found in \eqref{defn:RL:short} and \eqref{defn:Ca:short}.
From physical point of view, the FNBCs (R-L or Caputo) is the reflecting (no-flux) BCs, and the homogeneous classical Dirichlet BCs is the absorbing BCs~\cite{Baeumer2018JCAM}. For the R-L problem with FDBCs, although the physical meaning may not be clear, it is mathematically interesting, see \cite{ChenShenWang2016MathComp} for the one-sided FDEs or \cite{MaoChenShen16} for the Riesz FDEs. Therefore, in the present work we consider and analyze all four types of the aforementioned BCs.

It is well known that it is difficult to obtain  analytic forms of the solutions of  FDEs, hence efficient numerical methods are required. For the one-dimensional two-sided FDEs, there are several available numerical methods, for example, the finite difference method~\cite{MeerMM04, DengweihuaMC15}, the finite element method~\cite{ErvinRoop06, Wanghong13}, and the spectral method~\cite{LiXu10, MohsenKar13, MaoShen16, MaoChenShen16,  ErvinRoop18, MaoKar18} and references therein.
In the early works, the emphasis was on obtaining high accuracy by ignoring the issue of low regularity of the solution of FDEs, i.e.,  assuming that  the solution is smooth, for example, \cite{ErvinRoop06, LiXu10, DengweihuaMC15, MaoShen16}.  However, solutions of  fractional boundary value problems have endpoints singularities that limit the convergence rate of numerical discretizations significantly.  In order to resolve this issue, special treatments are required.
Jin et al. proposed   finite element approximations by using a regularity reconstruction  ~\cite{J.Z.2015ESIAM} or  regularity pickup \cite{J.L.Z.2016SIAMNUM}  to improve the convergence rate for one-sided FDEs; Mao and Shen developed a spectral element method using a geometric mesh to obtain spectral convergence with respect to the square root of of the number of degrees of freedom~\cite{MaoShen17Adv}. The spectral method using Jacobi poly-fractonomials, also known as general Jacobi functions, as basis function was first proposed in the  work~\cite{MohsenKar13} and in subsequent work~\cite{MaoChenShen16, MaoKar18}  by matching the singularities at the endpoints. A tunable spectral collocation method was also developed in \cite{ZMK2017SIAMJSC}.
Hesthaven et al.\cite{xu2014stable} considered local Dirichlet and Neumann BCs, and
proposed a multi-domain spectral penalty method for one-sided FDEs in  non-conservative form with Caputo fractional derivatives, but it does not preserve the positivity of solutions~\cite{Baeumer2018JCAM}.

However, most available numerical methods assume local Dirichlet BCs. In the present work, we aim to use spectral approximation to solve the two-sided FDEs \eqref{Intro1.1} with general BCs \eqref{IntroFDBC1.2}-\eqref{IntroFNBC1.5}. Unlike the case of a simple model problem with homogeneous Dirichlet BCs in which we can analyze the singularities at the endpoints (see \cite{MaoKar18}), the endpoint singularities of the solutions of the two-sided FDEs with {general BCs} used in the present work are not known, especially for the case of non-homogeneous BCs. This means that there is no suitable basis function that can approximate the solution well as in the case of \cite{MaoKar18}.
In the present work, we develop a spectral penalty method (SPM) for FDEs with two-sided fractional R-L and Caputo derivatives.
In particular, we formulate SPM by using  the Jacobi poly-fractonomial approximation for the R-L FDEs while using the polynomial approximation for the Caputo FDEs.
The penalty method for spectral approximations  was first introduced by  Gottlieb and Funaro \cite{funaro1991convergence} for collocation, and subsequently, several works appeared  employing  the penalty method to solve general boundary value problems (BVPs) for integer-order (see \cite{hesthaven1996stable,hesthaven1997stable, hesthaven2000spectral, black1992polynomial}).


The remainder of this article is structured as follows. We recall some basic notations and properties for fractional calculus and Jacobi poly-fractonomials and develop the spectral relationship between the fractional operators and the Jacobi poly-fractonomials in Section \ref{sec:pre}.
We establish in Section \ref{sec:wellpose} the well-posedness  for the weak problem of the conservative R-L problem with the FDBCs/FNBCs.  We formulate SPM for the conservative R-L and Caputo FDEs in Section \ref{sec:spm:RL}, where the poly-fractonomial approach is used for the R-L problem while the polynomial approach is used for the Caputo problem. We also address the  question of  \emph{coercivity} of SPM and provide sufficient conditions for  different types of fractional problems. Moreover we present  estimates of the penalty parameters and associated functions.  In Section \ref{sec:num}, we present several numerical examples to illustrate the proposed methods, demonstrating  that SPM  can deliver superior  accuracy  compared with a Petrov-Galerkin spectral tau method (PGS-$\tau$), an extension of the high accuracy method of \cite{MaoKar18}. Furthermore, we verify numerically the theoretical estimates for the sufficient conditions for \emph{coercivity} as well as  the estimates for  the penalty parameters. Finally, we present an application to the time dependent  fractional diffusion equation in Section \ref{sec:application}.
We conclude in Section \ref{sec:conclusion}.

\section{Preliminaries}\label{sec:pre}
In this section, we recall the basics of fractional integrals and derivatives,  and review some  relevant properties of the Jacobi poly-fractonomials. In particular, we introduce the spectral relationships between fractional operators and Jacobi polynomials.
\subsection{Fractional integrals and derivatives}
We begin by presenting the definitions of fractional integrals and derivatives.
Consider a generic interval $(a,b)$, let $\Gamma(\cdot)$ be the  usual Gamma function.
\begin{Defi}
(Fractional integrals and derivatives). For $\sigma \in R^{+}$, the left and right fractional integrals are defined respectively as \cite{podlubny1998fractional}
   \begin{equation*}
     \begin{split}
      \ilxa^{\sigma}v(x)=\frac{1}{\Gamma(\sigma)}\int_{a}^{x}\frac{v(y)}{(x-y)^{1-\sigma}}dy,\quad
   \irx^{\sigma}v(x)=\frac{1}{\Gamma(\sigma)}\int_{x}^{b}\frac{v(y)}{(y-x)^{1-\sigma}}dy,~x\in (a,b).
      \end{split}
   \end{equation*}
Thus, for $\sigma\in[k-1,k),\, k\in \mathbb{N}$, we define the left and right R-L fractional derivatives as
\begin{equation*}
  \dlxa^{\sigma}v(x)=D^{k}{ \ilxa^{k-\sigma}v(x)},\quad \drxa^{\sigma}v(x)=(-1)^{k}D^{k}{ \irxa^{k-\sigma}v(x)}
\end{equation*}
and  the left and right Caputo fractional derivatives  as
\begin{equation*}
  \cdlxa^{\sigma}v(x)={ \ilxa^{k-\sigma}D^{k}v(x)},\quad  \cdrxa^{\sigma}v(x)=(-1)^{k}{ \irxa^{k-\sigma}D^{k}v(x)},
\end{equation*}
where $D^{k}:=d^{k}/dx^{k}$.
\end{Defi}
The fractional integral operators satisfy the following  semigroup property: for $\sigma, \, \rho\ge 0$,
\begin{equation}\label{eqn:semig}
      \ilx^{\sigma}\; \ilx^{\rho} v(x) = \ilx^{\sigma+\rho} v(x),
      \quad \irx^{\sigma}\; \irx^{\rho} v(x) = \irx^{\sigma+\rho} v(x),
\end{equation}
and the adjoint property: for $\sigma\ge 0$,
\begin{equation}\label{eqn:adj}
    (\ilx^{\sigma}u, v) = (u, \irx^{\sigma} v).
\end{equation}
By the above two properties, we can deduce the following fractional integration by parts:
\begin{equation}\label{eqn:fintpart}
    (\ilx^{2\sigma}u, v) = (\ilx^{\sigma}u, \irx^{\sigma} v) =  (u, \irx^{2\sigma} v),\quad \sigma \ge 0.
\end{equation}

\subsection{Jacobi poly-fractonomials}
Let $P_{n}^{\mu,\nu}(x), ~\mu,\nu>-1, \; n\in \mathbb{N}$ be the classical Jacobi polynomial.
We now review the definition of Jacobi poly-fractonomials (also called general Jacobi functions) $J_{n}^{-\mu,-\nu}$ and discuss the spectral relationship for the two-sided fractional operators. The Jacobi poly-fractonomials are defined as follows: for all $x\in \Lambda$,
\begin{equation*}
  J_{n}^{-\mu,-\nu}(x):=(1-x)^{\mu}(1+x)^{\nu}P_{n}^{\mu,\nu}(x), ~\mu,\nu>-1, \; n\in \mathbb{N}.
\end{equation*}
%
The Jacobi poly-fractonomials $J_{n}^{-\mu,-\nu}(x),n\ge 0$ are orthogonal with respect to the weight function $\omega^{-\mu,-\nu}(x)$:
\begin{equation}\label{PreJa2.9}
\int_{\Lambda} J_{n}^{-\mu,-\nu}(x)J_{m}^{-\mu,-\nu}(x)\omega^{-\mu,-\nu}(x)dx=\gamma_{n}^{\mu,\nu}\delta_{mn},
\end{equation}
where $\gamma_{n}^{\mu,\nu}$ is given by the equation (3.88) in \cite{shen2011spectral}.
%

Moreover,
for $1<\alpha<2$, $2-\alpha\leq\nu,\mu<0$, $\mu+\nu+2-\alpha=0$, $0 \leq p \leq 1$, and $x \in (a,b)$, let
\begin{equation}\label{PreJa2.10}
  C_{\alpha, p}:=C(\alpha,\mu,\nu)=(sin\pi\mu+sin\pi\nu)/{sin\pi\alpha},
\end{equation}
and denote
\begin{equation}\label{PreJa2.11}
 \ilr^{\mu,\nu,\varrho}:=C_{\alpha,p}(p\ilp^{\varrho}+(1-p)\irp^{\varrho}) \text{ and }
 \dlr^{\mu,\nu,q}:=\frac{d^k}{dx^k}\ilr^{\mu,\nu,k-q}
\end{equation}
the two-sided fractional integral of order $\varrho$
and the  two-sided fractional R-L derivatives of order $q \in (k-1,k),k \in \mathbb{N}$, respectively.
See \cite{MaoKar18} for more details.

\begin{Lem}\label{lem2.1}
For a given $p,\, 0<p<1$ and $1<\alpha<2$, if $\alpha-2 <  \mu, \nu < 0$ and  $\mu, \nu, p$ satisfying
 \begin{equation}\label{PreJa2.13}
   \mu+\nu=\alpha-2,\; psin(\pi\mu) = (1-p){sin(\pi\nu)},
 \end{equation}
 then for $t\in(0,1)$ and $k=0,1,2,\ldots, $ we have that
 \begin{equation}\label{PreJa2.14}
   \mathcal{I}_{p,0,1}^{\mu,\nu,2-\alpha}t^{\nu}(1-t)^{\mu}t^{k}=\sum_{j=0}^{k}a_{k,j}t^{j},
 \end{equation}
 where
 \begin{equation}\label{eqn:akj}
   a_{k,j}=(-1)^{k}\frac{(-1)^{j}\Gamma(j+\alpha-1)\Gamma(\mu+1)}{\Gamma(\alpha-1-\nu-k+j)\Gamma(j+1)\Gamma(k+1-j)}.
 \end{equation}
 \end{Lem}
\begin{proof}
For $0<p<1$ and $g(t)=t^{\nu}(1-t)^{\mu}t^{k}$,  using the same argument as for the Lemma $5.1$ of  \cite{ErvinRoop18}, the following two equations hold:
\begin{equation*}
{}_{0}I_{t}^{2-\alpha}g(t)=\frac{\Gamma(1+\nu+k)}{\Gamma(3-\alpha+\nu+k)}t^{2-\alpha+\nu+k}{}_{2}F_{1}(1+\nu+k,-\mu; 3-\alpha+\nu+k, t),
\end{equation*}
\begin{equation*}
\begin{split}
  {}_{t}I_{1}^{2-\alpha}g(t)=\frac{\Gamma(-2+\alpha-\nu-k)}{\Gamma(-\nu-k)}t^{2-\alpha+\nu+k}{}_{2}F_{1}(1+\nu+k,-\mu; 3-\alpha+\nu+k, t)\\
  +\frac{\Gamma(\mu+1)\Gamma(2-\alpha+\nu+k)}{\Gamma(2-\alpha)\Gamma(3-\alpha+\nu+\mu+k)}{}_{2}F_{1}(-k,\alpha-1; -\nu-1+\alpha-k, t),
\end{split}
\end{equation*}
where
$ _{2}F_{1}(a_{1},b_{1};c_{1};x)=\sum_{j=0}^{\infty}\frac{(a_{1})_{j}(b_{1})_j}{(c_{1})_j}\frac{x^{j}}{j!} ~(|x|<1,~ a_{1}, b_{1}, c_{1}\in R,~ -c_{1} \notin N) $
is a hypergeometric function, and the rising  factorial in the Pochhammer symbol for $a_{1}\in R$ and~ $j \in N$ is defined by
\begin{equation*}
(a_{1})_{0}=1; ~(a_{1})_{j}:=a_{1}(a_{1}+1)\cdots(a_{1}+j-1)=\frac{\Gamma(a_{1}+j)}{\Gamma(a_{1})}, ~for ~j\geq 1.
\end{equation*}
Using the formula
  $\Gamma(1-z)\Gamma(z)=\frac{\pi}{sin{(\pi z)}}$
 gives
\begin{equation*}
 \frac{\Gamma(-2+\alpha-\nu-k)}{\Gamma(-\nu-k)}=-\frac{sin(\pi\nu)}{sin(\pi \mu)}\frac{\Gamma(1+\nu+k)}{\Gamma(3-\alpha+\nu+k)}.
\end{equation*}
By letting $p$ given by \eqref{PreJa2.13} and noting that $\Gamma(2-\alpha)\Gamma(\alpha-1)=-\frac{\pi}{sin(\pi \alpha)}$ and $\Gamma(2-\alpha+\nu+k)\Gamma(-1-\alpha-\nu-k)=(-1)^{k+1}\frac{\pi}{sin(\pi \mu)}$, we obtain the  equality \eqref{PreJa2.14}.
%
\end{proof}

Lemma \ref{lem2.1} implies that the two-sided fractional integral of Jacobi poly-fractonomials return polynomials; this is also true for the one-sided and Riesz fractional integrals, which are subcases of the general case~\cite{MohsenKar13, MaoChenShen16}.
\begin{Theo}\label{theo2.1}
For a given $p,\, 0<p<1$ and $1<\alpha<2$, if $\alpha-2 <\mu, \nu < 0$ and  $\mu, \nu, p$ satisfying
 the condition \eqref{PreJa2.13}, then for
$n=0,1,2,\ldots,$ it holds that
 \begin{equation}\label{PreJa2.16}
 \mathcal{ I}_{p,0,1}^{\mu,\nu,2-\alpha}\omega_{*}^{\mu,\nu}G_{n}(\mu,\nu,t)=\lambda_{n}G_{n}(\nu,\mu,t), \text{ where }
 \lambda_{n}=\frac{\Gamma(n+\alpha-1)}{\Gamma(n+1)},
 \end{equation}
where
$G_n(\mu,\nu, t)= P_n^{\mu,\nu}(x(t))$ is defined on interval $(0, 1)$ and  $x(t)=2t-1,\; \omega_{*}^{\mu,\nu}(t)=t^{\nu}(1-t)^{\mu}$.
\end{Theo}
\begin{proof}
The proof is similar to the one of Lemma 5.2 in  \cite{ErvinRoop18}. 
Let $\mathbb{P}_N$ be the space of polynomials of degree at most $N$, for $h(x)\in \mathbb{P}_{n-1}$, $(G_{n}(\nu,\mu,t),h)_{\omega_{*}^{\nu,\mu}}=0$,
$(I_{p,0,1}^{\mu,\nu,2-\alpha}\omega_{*}^{\mu,\nu}G_{n}(\mu,\nu,t), h)_{\omega_{*}^{\nu,\mu}}=0$, hence, $$I_{p,0,1}^{\mu,\nu,2-\alpha}\omega_{*}^{\mu,\nu}G_{n}(\mu,\nu,t) =CG_{n}(\nu,\mu,t)$$
with $C$ is a constant.
Since the coefficient of $x^{n}$ in $G_{n}(\mu,\nu,t)$ and $G_{n}(\nu,\mu,t)$ is $\frac{\Gamma(2n+\mu+\nu+1)}{n!\Gamma(n+\mu+\nu+1)}$, then from $a_{n,n}$, we get
\begin{equation*}
 \frac{\Gamma(n+\alpha-1)}{\Gamma(n+1)}=  \lambda_{n}.
\end{equation*}
\end{proof}

Due to Theorem \ref{theo2.1} and Lemma \ref{lem2.1}, we can get the following results:
\begin{Theo}\label{theo2.2}
For a given $p,\, 0\le p\le 1$ and $1<\alpha<2$, if $\alpha-2 \leq  \mu, \nu \le 0$ and  $\mu, \nu, p$ satisfying
the condition \eqref{PreJa2.13}, then
 for $x \in (-1,1)$, $n=0,1,2,\ldots,$ it holds that
 \begin{equation}\label{PreJa2.18}
  \mathcal{I}_{p,-1,1}^{\mu,\nu,2-\alpha}J_{n}^{-\mu,-\nu}(x)=\lambda_{n}P_{n}^{\nu,\mu}(x), ~\text{where} \quad
  \lambda_{n}=\frac{\Gamma(n+\alpha-1)}{\Gamma(n+1)},
 \end{equation}
 and
 \begin{equation}\label{PreJa2.19}
  \mathcal{D}_{p,-1,1}^{\mu,\nu,k+\alpha-2}J_{n}^{-\mu,-\nu}(x)=\tilde{C}_{\a}P_{n-k}^{\nu+k,\mu+k}(x),~ \text{where} \quad
     \tilde{C}_{\a}=\frac{\Gamma(n+k+\alpha-1)}{2^k\Gamma(n+1)}.
\end{equation}
\end{Theo}
\begin{proof}
For $0<p<1$,
by the transformation $x:=2t-1$ in equation \eqref{PreJa2.16}, we deduce \eqref{PreJa2.18} by
\begin{equation*}
\begin{aligned}
 I_{p,-1,1}^{\mu,\nu,2-\alpha}J_{n}^{-\mu,-\nu}(x) &=I_{p,-1,1}^{\mu,\nu,2-\alpha}\omega^{\mu,\nu}P_{n}^{\mu,\nu}(x)
 =I_{p,0,1}^{\mu,\nu,2-\alpha}\omega_{*}^{\mu,\nu}G_{n}(\mu,\nu,t)\\
 &=\lambda_{n}G_{n}(\nu,\mu,t)
 =\lambda_{n}P_{n}^{\nu,\mu}(x).
\end{aligned}
\end{equation*}
For $p = 1$, we set $\mu = \alpha-2,\, \nu =0$ while for $p = 1$ we set $\mu =0 ,\, \nu =\alpha-2$. Obviously, $\mu,\nu$ satisfy \eqref{PreJa2.13}, then we can derive \eqref{PreJa2.18} for $p = 0,1$ by using \cite[Equations (2.34) and (2.35)]{ChenShenWang2016MathComp}.
Equation \eqref{PreJa2.19} can be deduced from \eqref{PreJa2.11}, \eqref{PreJa2.18} and from equation (3.101) in \cite{shen2011spectral}.
\end{proof}

For the sake of simplicity, we denote
\begin{equation}\label{defn:RL:short}
    \ix^{\varrho}:=\mathcal{I}_{p, a, b}^{\mu,\nu,\varrho},\text{ and }\dx^{\varrho}:=\mathcal{D}_{p, a, b}^{\mu,\nu,\varrho}
\end{equation}
be the R-L two-sided fractional integral and derivative if no confusion arises.
Similarly, we can define the two-sided Caputo fractional derivative
\begin{equation}\label{defn:Ca:short}
    \dxc^{\varrho}:=C_{\a,p}(p\,\cdlx^{\varrho} - (1-p)\,\cdrx^{\varrho}).
\end{equation}

\section{Well-posedness}\label{sec:wellpose}
Before establishing the discretization scheme for the fractional problems \eqref{Intro1.1} with the general BCs \eqref{IntroFDBC1.2}-\eqref{IntroFNBC1.5},
we first show the well-posedness of the continuous weak problem.
For the case of conservative Caputo problem, the well-posedness results have been shown in \cite{Ma.J2017JSC}. Also, for the one-sided model problem without the reaction term, i.e.,  $c \equiv 0$, Wang and his collaborators  showed the well-posedness of the fractional Dirichlet boundary problem in the conservative Caputo sense (cf. \cite{WangYangZhu2014SIAMNUM}) and the fractional Neumann boundary problem in both conservative Caputo and R-L sense (cf. \cite{WangYang2017FCAA}).
We discuss in this section the well-posedness of the case of two-sided conservative R-L fractional problem.
\subsection{Fractional integral/derivative spaces}

For $\sigma\ge 0$, define the fractional integral spaces associated with the left and right fractional integrals:
\begin{equation*}
    J_{l}^{-\sigma}(\Lambda) := \left\{ v: \ilx^{\sigma}v \in L^2(\Lambda)\right\} \text{ and } J_{r}^{-\sigma}(\Lambda) := \left\{ v: \irx^{\sigma}v \in L^2(\Lambda)\right\}
\end{equation*}
%
with norms
\begin{equation}\label{defn:fintspl:norm}
\begin{aligned}
    &\|v\|^2_{J_{l}^{-\sigma}(\Lambda)} := (\ilx^{\sigma}v, \ilx^{\sigma}v) = \|\ilx^{\sigma}v\|^2_{L^2(\Lambda)}  \text{ and } \\
    & \|v\|^2_{J_{r}^{-\sigma}(\Lambda)} := (\irx^{\sigma}v, \irx^{\sigma}v) = \|\irx^{\sigma}v\|^2_{L^2(\Lambda)},
\end{aligned}
\end{equation}
%
respectively.
Moreover, we define the following fractional integral space and norm associated with the fractional integral $\ix^{\sigma}$ for $\sigma\ge 0$:
\begin{equation}\label{defn:fintsps}
    J_{p}^{-\sigma}(\Lambda) := \left\{ v: \ix^{\sigma}v \in L^2(\Lambda) \right\}
\; \text{with }
    \|v\|^2_{J_{p}^{-\sigma}(\Lambda)} := (\ix^{\sigma}v, \ix^{\sigma}v) = \|\ix^{\sigma}v\|^2_{L^2(\Lambda)}.
\end{equation}
When $p = 1$ (resp. $p=0$), the fractional integral space $J_{p}^{-\sigma}(\Lambda)$ reduces to the space $J_{l}^{-\sigma}(\Lambda)$ (resp. $J_{r}^{-\sigma}(\Lambda)$).


We now show that $J_{p}^{-\sigma}(\Lambda),\, \sigma \ge 0$ is a Hilbert space. To do this, we only need to verify that $\|v\|_{J_{p}^{-\sigma}(\Lambda)} = 0$ if and only if $v(x) =  0$. On the one hand, if $v(x) = 0$, obviously, $\|v\|_{J_{p}^{-\sigma}(\Lambda)} = 0$. On the other hand,  by the spectral relation \eqref{PreJa2.18}, we have that there exists a unique pair of $\mu,\nu$ satisfying \eqref{PreJa2.13} (where $2-\alpha$ replaced by $\sigma$), such that $v \in L_{\omega^{-\mu,-\nu}}^2(\Lambda)$ and $\|v\|_{L_{\omega^{-\mu,-\nu}}^2(\Lambda)} \propto \|\ix^{\sigma}v\|_{L_{\omega^{\nu,\mu}}^2(\Lambda)}$. Thus, $\|v\|_{J_{p}^{-\sigma}(\Lambda)} = 0$ gives $\ix^{\sigma}v(x) = 0$ and then $\|v\|_{L_{\omega^{-\mu,-\nu}}^2(\Lambda)} = 0$. Then $v(x) = 0$ follows the completeness of the space $L_{\omega^{-\mu,-\nu}}^2(\Lambda)$ (see the argument in \cite[Section 3.4]{MaoKar18}).

More technical results  for the fractional integral spaces that would be used in next subsection are presented in Appendix  \ref{sec:Tech:RePr}.

As a consequence of the fractional integral space $J_{p}^{\delta -1}(\Lambda)$ for $0<\delta<1$, we can define the  following Riemann-Liouville fractional derivative space and norm:
\begin{equation}\label{Frac:dspace}
    H_{RL}^{\delta}(\Lambda) := \left\{v:  v \in J_{p}^{\delta -1}(\Lambda),\; \dx^{\delta}v \in L^2(\Lambda)\right\}
,\;
    \|v\|_{H_{RL}^{\delta}(\Lambda)}^2 := \|v\|_{J_{p}^{\delta -1}(\Lambda)}^2 + \|\dx^{\delta}v \|_{L^2(\Lambda)}^2.
\end{equation}
Obviously, the fractional derivative space $H_{RL}^{\delta}(\Lambda)$ is a complete space.
We then define the space $H_{RL}^{\delta, 0}(\Lambda),\, 0<\delta<1$ as follows:
\begin{equation*}
    H_{RL}^{\delta, 0}(\Lambda): = \left\{v:v \in H_{RL}^{\delta}(\Lambda),\; \ix^{1-\delta} v(\pm 1) = 0\right\}.
\end{equation*}

\subsection{Weak problems and their well-posedness}
In this subsection, we give the weak formulations of the fractional Dirichlet boundary  problem \eqref{Intro1.1}-\eqref{IntroFDBC1.2} and the fractional Neumann boundary problem \eqref{Intro1.1}-\eqref{IntroFNBC1.3}, and prove their  well-posedness.

\subsubsection{Fractional Dirichlet boundary problem}
For the sake of simplicity, for the Dirichlet boundary problem, we only discuss the homogeneous BCs, i.e.,
\begin{equation}\label{fDBC:Homo}
    \ix^{2-\alpha} u(-1) = 0, \quad \ix^{2-\alpha} u(1) = 0.
\end{equation}
Actually, the non-homogeneous problem can be transferred into a homogeneous problem by using the lifting technique.
The weak formulation of problem \eqref{Intro1.1}-\eqref{fDBC:Homo} is obtained by multiplying the fractional integral of the test function $v(x)$ where $v(x) \in H_{RL}^{\alpha - 1,0}(\Lambda)$, i.e., $\ix^{2-\alpha}v(x)$, on both sides of equation \eqref{Intro1.1} and integrating over $\Lambda$.
Then, the weak formulation of \eqref{Intro1.1}-\eqref{fDBC:Homo} is to find $u(x) \in H_{RL}^{\alpha - 1,0}(\Lambda)$, such that
\begin{equation}\label{wk:fDBC}
    \mathcal{A}(u,v) = \mathcal{F}(v) \quad \forall v(x) \in H_{RL}^{\alpha - 1, 0}(\Lambda),
\end{equation}
where the bilinear form $\mathcal{A}(\cdot,\cdot)$ and the linear functional $\mathcal{F}(v)$ are, respectively, given by
\begin{eqnarray}
  \mathcal{A}(u,v) &: =& c (u,\ix^{2-\alpha}v ) + (\dx^{\alpha-1}u, \dx^{\alpha-1}v), \label{wk:F:d}\\
  \mathcal{F}(v) &:=& \langle f,\ix^{2-\alpha} v\rangle. \label{wk:F:d:F}
\end{eqnarray}

\begin{Lem}\label{cont:coer:d}
The bilinear form $\mathcal{A}(\cdot,\cdot)$ is continuous and coercive in $H_{RL}^{\alpha-1}(\Lambda)\times H_{RL}^{\alpha-1}(\Lambda)$ , i.e.,
\begin{equation}\label{dbc:cont}
    |\mathcal{A}(u,v)| \lesssim \|u\|_{H_{RL}^{\alpha-1}(\Lambda)} \|v\|_{H_{RL}^{\alpha-1}(\Lambda)}; \;
    \mathcal{A}(u,u) \gtrsim \|u\|_{H_{RL}^{\alpha-1}(\Lambda)}^2 \quad
    \forall\,  u,\,v\, \in H_{RL}^{\alpha-1}(\Lambda).
\end{equation}
%
%
\end{Lem}
\begin{proof}
By the Cauchy-Schwarz inequality, we have
\begin{equation}\label{dbc:cont:pf1}
  |\mathcal{A}(u,v)| \le  c \|u\|_{L^2(\Lambda)} \|v\|_{J_{p}^{\alpha-2}(\Lambda)} + \|\dx^{\alpha-1}u\|_{L^2(\Lambda)} \|\dx^{\alpha-1}v\|_{L^2(\Lambda)}
\end{equation}
Moreover, by letting $q= 2$ in Lemma \ref{lem:ixbdd}, we deduce
\begin{equation*}
    \|u\|_{J_{p}^{\alpha-2}(\Lambda)}^2 \le C_1\|u\|_{J_{l}^{\alpha-2}(\Lambda)}^2 +C_2\|u\|_{J_{r}^{\alpha-2}(\Lambda)}^2 \le C \|u\|_{L^2(\Lambda)}^2.
\end{equation*}
%
This means that  $L^2(\Lambda) \subset J_{l}^{\alpha-2}(\Lambda)$, $ L^2(\Lambda) \subset J_{r}^{\alpha-2}(\Lambda)$ and $L^2(\Lambda) \subset J_{p}^{\alpha-2}(\Lambda)$.
On the other hand, $\mathcal{H}^{\alpha-1}(\Lambda) \subset L^2(\Lambda)$ where
\begin{equation*}
   \mathcal{H}^{\alpha-1}(\Lambda) := \left\{w:\|w\|_{L^2(\Lambda)} + \|\dx^{\alpha-1}w\|_{L^2(\Lambda)}< \infty \right\}.
\end{equation*}
Then, by the interpolation theorem, $\forall \, \eta>0$, there exists a constant $C_{\eta}$ such that
\begin{equation*}
    \|u\|_{L^2(\Lambda)} \le \eta \|u\|_{\mathcal{H}^{\alpha-1}(\Lambda)} + C_{\eta}\|u\|_{J_{p}^{\alpha-2}(\Lambda)}.
\end{equation*}
Letting $\eta = 1/2$, we arrive at
\begin{equation*}
    \|u\|_{L^2(\Lambda)} \le 1/2 (\|u\|_{L^2(\Lambda)} + \|\dx^{\alpha-1}u\|_{L^2(\Lambda)}) + C_{1/2}\|u\|_{J_{p}^{\alpha-2}(\Lambda)},
\end{equation*}
which yields
\begin{equation*}
    \|u\|_{L^2(\Lambda)} \le C \|u\|_{H_{RL}^{\alpha-1}(\Lambda)}.
\end{equation*}
Thus, we obtain from the above estimate and \eqref{dbc:cont:pf1} that the bilinear form $\mathcal{A}(\cdot,\cdot)$ is continuous in  $H_{RL}^{\alpha-1}(\Lambda)\times H_{RL}^{\alpha-1}(\Lambda)$, i.e., the estimate \eqref{dbc:cont} holds true.

If $c > 0$, the coercivity can be readily derived by the fractional integration by parts \eqref{eqn:fintpart} and letting $s = 1-\alpha/2,\; t = 2-\alpha$ in the estimate \eqref{eqn:intcoer}. For the case of $c \equiv 0$, the coercivity can be obtained by applying the Poincar{\'e} inequality to the function $\ix^{2-\alpha} u(x)$, namely,
\begin{equation*}
    \|u\|_{{J_{p}^{\alpha-2}(\Lambda)}} \le C\|\dx^{\alpha-1}u\|_{L^2(\Lambda)} \quad \forall \, u\in H_{RL}^{\alpha-1,0}(\Lambda).
\end{equation*}
\end{proof}

For the linear functional $\mathcal{F}(v)$ given by \eqref{wk:F:d:F}, we have for $f\in H^{-1}(\Lambda)$,
\begin{equation*}
    \mathcal{F}(v)= \langle f,\ix^{2-\alpha}v \rangle \le \|f\|_{H^{-1}(\Lambda)} \|v\|_{H_{RL}^{\alpha-1}(\Lambda)}.
\end{equation*}
We then have the following result:
\begin{Lem}\label{lem:flinear:d}
Let $f\in H^{-1}(\Lambda)$.
The linear functional $\mathcal{F}(v)$ given by \eqref{wk:F:d:F} is continuous on $H_{RL}^{\alpha-1}(\Lambda)$.
\end{Lem}

Using Lemma \ref{cont:coer:d}, \ref{lem:flinear:d} and the Lax-Milgram theorem, we have the well-posedness of the weak problem \eqref{wk:fDBC}, namely, we have the  the following result:
\begin{Theo}\label{thm:well-posed:d}
For $f\in H^{-1}(\Lambda)$, the weak problem \eqref{wk:fDBC} admits a unique solution $u(x)\in H_{RL}^{\alpha - 1, 0}(\Lambda)$ satisfying
\begin{equation}\label{eqn:uesti:d}
    \|u\|_{H_{RL}^{\alpha-1}(\Lambda)} \lesssim  \|f\|_{H^{-1}(\Lambda)}.
\end{equation}
\end{Theo}
The estimate \eqref{eqn:uesti:d} follows from
\begin{equation*}
    \|u\|_{H_{RL}^{\alpha-1}(\Lambda)}^2 \lesssim \mathcal{A}(u,u) = \langle f,\ix^{2-\alpha}u\rangle  \le \|f\|_{H^{-1}(\Lambda)} \|u\|_{H_{RL}^{\alpha-1}(\Lambda)}.
\end{equation*}

\subsubsection{Fractional Neumann boundary problem}
We now consider the fractional Neumann boundary problem \eqref{Intro1.1}-\eqref{IntroFNBC1.3}. In this case, we assume that $c$ is positive away from $0$. For the case of $c\equiv 0$, we can add the condition of mass conservation, but we will not discuss this case here.
The weak formulation of problem \eqref{Intro1.1} and \eqref{IntroFNBC1.3} is obtained in the same way as that for the fractional Dirichlet boundary problem. Then, we obtain the weak formulation of \eqref{Intro1.1} and \eqref{IntroFNBC1.3}: find $u(x) \in H_{RL}^{\alpha - 1}(\Lambda)$, such that
\begin{equation}\label{wk:fNBC}
    \mathcal{A}(u,v) = \mathrm{F}(v) \quad \forall v(x) \in H_{RL}^{\alpha - 1}(\Lambda),
\end{equation}
where the bilinear form $\mathcal{A}(u,v)$ is, again, given by \eqref{wk:F:d} and  the linear functional $\mathrm{F}(v)$ in this case is given by
\begin{equation}\label{wk:F:n}
    \mathrm{F}(v):=\langle f,\ix^{2-\alpha}v\rangle  + g_{2}\,\ix^{2-\alpha} v(1) - g_{1}\,\ix^{2-\alpha}v(-1).
\end{equation}

The continuity and coercivity of  the bilinear form $\mathcal{A}(\cdot,\cdot)$ are given in Lemma \ref{cont:coer:d}. Next we prove the continuity of the linear functional $\mathrm{F}(v)$ for the fractional Neumann boundary problem.
\begin{Lem}
Let $f\in H^{-1}(\Lambda)$.
The linear functional $\mathrm{F}(v)$ given by \eqref{wk:F:n} is continuous on $H_{RL}^{\alpha-1}(\Lambda)$.
\end{Lem}
\begin{proof}
%
%
We can estimate $\mathrm{F}(v)$ in the following way:
\begin{equation*}
\begin{aligned}
    \mathrm{F}(v) &= \langle f,\ix^{2-\alpha}v\rangle + g_{2}\, \ix^{2-\alpha} v(1) - g_{1}\, \ix^{2-\alpha}v(-1)\\
    & \le \|f\|_{H^{-1}(\Lambda)} \|v\|_{H_{RL}^{\alpha-1}(\Lambda)} +  |g_{2}|\cdot|\ix^{2-\alpha} v(1)| + |g_{1}|\cdot|\ix^{2-\alpha} v(-1)| \\
    & \le \big(\|f\|_{H^{-1}(\Lambda)} + C(|g_1|+|g_{2}|) \big)\|v\|_{H_{RL}^{\alpha-1}(\Lambda)}.
\end{aligned}
\end{equation*}
Hence, the linear functional $\mathrm{F}(v)$ is continuous on $H_{RL}^{\alpha-1}(\Lambda)$.
\end{proof}

Again, by using the Lax-Milgram lemma, we can establish the well-posedness of the weak problem \eqref{wk:fNBC}, namely, we arrive at the following Theorem:
%
\begin{Theo}\label{thm:well-posed:n}
For $f\in H^{-1}(\Lambda)$,
the weak problem \eqref{wk:fNBC} admits a unique solution $u(x)\in H_{RL}^{\alpha - 1}(\Lambda)$ satisfying
\begin{equation*}
    \|u\|_{H_{RL}^{\alpha-1}(\Lambda)} \lesssim  \|f\|_{H^{-1}(\Lambda)} + C(|g_1|+|g_{2}|),
\end{equation*}
where $C$ is a constant.
\end{Theo}

\section{Spectral penalty method (SPM)}\label{sec:spm:RL}
We now consider the spectral approximation to the solution of the fractional problem \eqref{Intro1.1}. One possibility is to extend the method of \cite{MaoKar18} and formulate a PGS-$\tau$ as presented in Appendix  \ref{sec:app:tau} to solve the general non-homogeneous fractional  boundary problem.
However, as we will see, the accuracy  of PGS-$\tau$ is not as high due to the limited regularity of the solution of fractional problems with general non-homogeneous boundaries. Thus, we formulate SPM to discretize the fractional problem \eqref{Intro1.1}  with general BCs \eqref{IntroFDBC1.2}-\eqref{IntroFNBC1.5}.

\subsection{SPM for conservative two-sided R-L FDEs}
We first introduce SPM for  the two-sided conservative R-L FDE \eqref{Intro1.1} with the FDBCs \eqref{IntroFDBC1.2} or the  FNBCs \eqref {IntroFNBC1.3}.
In this case, we shall use the poly-fractonomials, i.e., $J_k^{-\mu,-\nu},k=0,1,\ldots$,  introduced in Section \ref{sec:pre} to approximate the solutions.
To this end, we  introduce some notations. Let $\omega >0$ be a generic weight function
and $\mathbb{P}_N$ be the space of polynomials of degree at most $N$. In addition, let $\mu, \nu$ be two real number satisfying the condition \eqref{PreJa2.13}.
We define the finite-dimensional  space:
\begin{equation*}
\mathcal{F}_{N}^{-\mu,-\nu}:=\left\{\phi=(1-x)^{\mu}(1+x)^{\nu}\varphi: \varphi \in \mathbb{P}_N \right\}=span\left\{J_{n}^{-\mu,-\nu}: 0\leq n \leq N \right\}.
\end{equation*}

\subsubsection{Numerical implementation of SPM}

By multiplying $\ix^{2-\a}v$ on both sides of \eqref{Intro1.1} and introducing the penalty parameters $\rho_{\pm}$ and the penalty functions $Q_N^{\pm}(x)$,  we have
the spectral penalty scheme, more general in a weighted sense, for \eqref{Intro1.1}-\eqref{IntroFDBC1.2} or \eqref{Intro1.1}-\eqref{IntroFNBC1.3}: find $u_{N} \in \mathcal{F}_{N}^{-\mu,-\nu}$ such that
\begin{equation}\label{RLPG}
\begin{aligned}
  &\mathcal{A}_P^{R-L}(u_N,v,\rho_{\pm}, Q_N^{\pm}) \\
  = & \langle f,\ix^{2-\a}v \rangle_{\omega}+\rho_{-}g_1(Q_N^{-},\ix^{2-\a}v)_{\omega}+\rho_{+}g_2( Q_N^{+},\ix^{2-\a}v)_{\omega} \quad \forall  v\in \mathcal{F}_{N}^{-\mu,-\nu},
\end{aligned}
\end{equation}
where the bilinear form $\mathcal{A}_P^{R-L}(\cdot,\cdot,\cdot,\cdot)$ is given by
\begin{equation*}
  \begin{split}
  \mathcal{A}_P^{R-L}(u,v,\rho_{\pm}, Q_N^{\pm}) : = -(\dx^{\alpha}u,\ix^{2-\a}v)_{\omega}+c(u,\ix^{2-\a}v)_{\omega} &\\
  +\rho_{-}\,\mathcal{B}_{-}u(-1)\,(Q_N^{-},\ix^{2-\a}v)_{\omega} +\rho_{+}\,\mathcal{B}_{+}u(1)\,(Q_N^{+},\ix^{2-\a}v)_{\omega}&
  \end{split}
\end{equation*}
and  $\mathcal{B}_{\pm}u(\pm1)=\ix^{2-\alpha}u(\pm1)$ for FDBCs \eqref{IntroFDBC1.2} while $ \mathcal{B}_{\pm}u(\pm1)=\dx^{\alpha-1}u(\pm1)$ for FNBCs \eqref{IntroFNBC1.3}. $Q_N^{\pm}$ and $\rho_{\pm}$ are  to be determined.

By taking
$u_N(x)=\sum_{k=0}^{N}\tilde{u}_{k}J_{k}^{-\mu,-\nu}(x)$,
letting the test functions  be $J_{i}^{-\mu,-\nu}(x)$, $0\leq i\leq N$, and denoting $\varphi_{i}(x) = \ix^{2-\a}J_{i}^{-\mu,-\nu}(x)$, we obtain the linear system
\begin{equation}\label{RLLMS}
  (-S + cM +B)U=\widehat{F}+\widetilde{F},
\end{equation}
where $U=(\tilde{u}_0,\tilde{u}_1,\cdots,\tilde{u}_{N})^T$ and
\begin{equation*}
\begin{aligned}
  &S=(s_{ik})_{i,k=0}^{N}, \; s_{ik}=(\mathcal{D}_{p}^{\alpha}J_k^{-\mu,-\nu},\varphi_i)_{\omega};
  \quad M=(m_{ik})_{i,k=0}^{N}, \;  m_{ik}=(J_k^{-\mu,-\nu},\varphi_{i})_{\omega};\\
  &B=(b_{ik}^{+} + b_{ik}^{-})_{i,k=0}^{N}, \;  b_{ik}^{\pm}=\rho_{\pm}(\mathcal{B}_{\pm}J_k^{-\mu,-\nu})(\pm 1) (Q_N^{\pm},\varphi_{i})_{\omega};\\
  & \widehat{F}=(\hat{f}_{0},\cdots,\hat{f}_{N})^{T},\; \hat{f}_{i}=\langle f,\varphi_{i}\rangle_{\omega}; \\
  &\widetilde{F}=(\tilde{f}_0,\cdots,\tilde{f}_{N})^{T},\; \tilde{f}_{i}= \rho_{+}g_2(Q_N^{+},\varphi_{i})_{\omega} +\rho_{-}g_1(Q_N^{-},\varphi_{i})_{\omega}.
\end{aligned}
\end{equation*}
We point out that the value of $\mathcal{D}_{p}^{\alpha}J_k^{-\mu,-\nu}$ and $\ix^{2-\a}J_k^{-\mu,-\nu}$ can be directly obtained by the spectral relationship \eqref{PreJa2.18} and \eqref{PreJa2.19}, and then all the integrals can be calculated using the Gauss quadrature.

\subsubsection{A sufficient condition for the coercivity of \eqref{RLPG}}
Once we have SPM \eqref{RLPG}, then we pose the  question on how to choose the parameters $\rho_{\pm}$ and functions $Q_N^{\pm}$. The crucial idea for choosing $\rho_{\pm}$ and $Q_N^{\pm}$ is to obtain the \emph{coercivity} of SPM \eqref{RLPG}. We now give a \emph{sufficient condition} for the \emph{coercivity} of SPM \eqref{RLPG}.

We first consider the case of FDBCs \eqref{IntroFDBC1.2}. In this case, $\mathcal{B}_{-}u_{N}(-1)= \mathcal{I}_{p}^{2-\alpha}u_{N}(-1)$, $ \mathcal{B}_{+}u_{N}(1)= \mathcal{I}_{p}^{2-\alpha}u_{N}(1)$, and we give the sufficient condition for the \emph{coercivity} with $c= 0$.
\begin{Theo}\label{theo4.1}
Let $c = 0$, $u_{N}$ be the solution of the penalty scheme \eqref{RLPG}
and  $\omega(x)=\omega^{\tilde{\alpha},\tilde{\beta}}(x) ~(-1<\tilde{ \alpha},\tilde{\beta} < 1)$ be the Jacobi type weight function.
Then
\begin{eqnarray}\label{CoerFDBC}
   \mathcal{A}_P^{R-L}(u_N,u_N,\rho_{\pm}, Q_N^{\pm}) \gtrsim |W_N|_{1,\omega}^2
\end{eqnarray}
provided
\begin{equation}\label{cond:coe:FDBC}
    \rho_{-}Q_{N}^{-}(x)=D^{2}P_{N}^{-}(x),\quad \rho_{+}Q_{N}^{+}(x)=D^{2}P_{N}^{+}(x),
\end{equation}
where
\begin{equation}\label{eqn:PNpm}
    P_N^{-}(x)=\frac{(1-x)P_{N+1}^{\tilde{\alpha}+1, \tilde{\beta}}(x)}{2P_{N+1}^{\tilde{\alpha}+1, \tilde{\beta}}(-1)},\; P_N^{+}(x)=\frac{(1+x)P_{N+1}^{\tilde{\alpha}, \tilde{\beta}+1}(x)}{2P_{N+1}^{\tilde{\alpha}, \tilde{\beta}+1}(1)}
\end{equation}
and $W_{N}(x)=\ix^{2-\alpha}u_{N}(x)-\ix^{2-\alpha}u_{N}(-1)\cdot P_{N}^{-}(x) - \ix^{2-\alpha}u_{N}(1)\cdot P_{N}^{+}(x)$.
\end{Theo}
\begin{proof}
By the definition of $W_N$, we have
\begin{equation*}
\begin{aligned}
 &-(\mathcal{D}_{p}^{\alpha}u_{N},\mathcal{I}_{p}^{2-\alpha}u_{N})_{\omega} \\
= &-(D^{2}W_{N},W_{N})_{\omega}
 -\mathcal{I}_{p}^{2-\alpha}u_{N}(-1)(D^{2}W_{N},P_{N}^{-})_{\omega}-\mathcal{I}_{p}^{2-\alpha}u_{N}(1)(D^{2}W_{N},P_{N}^{+})_{\omega}\\
& -\mathcal{I}_{p}^{2-\alpha}u_{N}(-1)(D^2P_{N}^{-},\mathcal{I}_{p}^{2-\alpha}u_{N})_{\omega}
 -\mathcal{I}_{p}^{2-\alpha}u_{N}(1)(D^2P_{N}^{+},\mathcal{I}_{p}^{2-\alpha}u_{N})_{\omega}.
\end{aligned}
\end{equation*}
%
For the polynomial $P_N^-$ given by \eqref{eqn:PNpm}, it can be written as~\cite[Theorem 3.19]{shen2011spectral}
\begin{equation}\label{PN:FDBC}
P_{N}^{-} = \frac{1}{P_{N+1}^{\tilde{\alpha}+1,\tilde{\beta}}(-1)}\frac{1}{2N+4+\tilde{\alpha}+\tilde{\beta}} \left((N+2+\tilde{\alpha})P_{N+1}^{\tilde{\alpha},\tilde{\beta}} - (N+2)P_{N+2}^{\tilde{\alpha},\tilde{\beta}} \right).
\end{equation}
Using the orthogonality and the fact that $D^{2}W_{N} \in \mathbb{P}_{N}$, we deduce $(D^{2}W_{N},P_{N}^{-})_{\omega} = 0$. We can also obtain $(D^{2}W_{N},P_{N}^{+})_{\omega} = 0$ by using the same argument.
Therefore, by the above two equations we can deduce that, by providing the condition \eqref{cond:coe:FDBC}, we have
\begin{equation*}
    \mathcal{A}_P^{R-L}(u_N,u_N,\rho_{\pm}, Q_N^{\pm}) = -(D^{2}W_{N},W_{N})_{\omega}.
\end{equation*}
Note that, $P_{N}^{+}(-1)=P_{N}^{-}(1)=0$ and $P_{N}^{+}(1)=P_{N}^{-}(-1)=1$ gives $W_{N}(\pm1)=0$, then the estimate \eqref{CoerFDBC} holds true by using the above equation and Lemma 3.5 in \cite{shen2011spectral}.
\end{proof}

\begin{Rem}\label{rem4.1}
Note that for the condition \eqref{cond:coe:FDBC}, the penalty parameter $\rho_{+}$ (resp. $\rho_{-}$) and the function $Q_N^{+}$ (resp. $Q_N^{-}$) are associated with each other. Thus, to estimate  the penalty parameters and functions, we shall estimate the combinations, namely $\rho_{\pm}Q_{N}^{\pm}$. By  \eqref{PN:FDBC} and the equation (3.101) in \cite{shen2011spectral}, we deduce that
\begin{equation*}
\begin{aligned}
  D^{2}P_{N}^{-} =& \frac{1}{(2N+4+\tilde{\alpha}+\tilde{\beta}) P_{N+1}^{\tilde{\alpha}+1,\tilde{\beta}}(-1)} \times \\
  &\left((N+2+\tilde{\alpha})d_{N+1,2}^{\tilde{\alpha},\tilde{\beta}}P_{N-2}^{\tilde{\alpha}+2,\tilde{\beta}+2}-(N+2)d_{N+2,2}^{\tilde{\alpha},\tilde{\beta}}P_{N}^{\tilde{\alpha}+2,\tilde{\beta}+2} \right).
\end{aligned}
\end{equation*}
According to the following estimate
\begin{equation}\label{CoerFDBC4.3}
  \frac{\Gamma(n+\tilde{a})}{\Gamma(n+\tilde{b})} \sim n^{\tilde{a}-\tilde{b}},\; \tilde{a},~\tilde{b} \in \mathbb{R},\, n \in \mathbb{N}, n+\tilde{a}>1, \text{ and } n+\tilde{b}>1,
\end{equation}
we have
$|P_{N+1}^{\tilde{\alpha}+1,\tilde{\beta}}(-1)|\sim N^{\tilde{\beta}},\; d_{N+1,2}^{\tilde{\alpha},\tilde{\beta}}\sim N^{2},\; d_{N+2,2}^{\tilde{\alpha},\tilde{\beta}}\sim N^{2}$, moreover, following the Theorem 3.24 in \cite{shen2011spectral}, we have  $\max\limits_{x\in \Lambda}|P_{N-2 \text{ or } N}^{\tilde{\alpha}+2,\tilde{\beta}+2}(x)|\sim N^{q}$ where $q= \max(\tilde{\alpha},\tilde{\beta}) +2$.
Thus, we arrive at $\rho_{-}Q_N^{-} = D^{2}P_{N}^{-} = O(N^{2+q-\tilde{\beta}})$. Similarly, we can obtain $\rho_{+}Q_N^{+} = D^{2}P_{N}^{+} = O(N^{2+q-\tilde{\alpha}})$.
Observe that the condition \eqref{cond:coe:FDBC} requires that the equality holds, however, the numerical results (see the right plot of Figure \ref{fig:minieig:allalpha} in Section \ref{sec:num}) show that the coercivity can be fulfilled by fixing $Q_{N}^{\pm} = O(1)$ and letting
\begin{equation*}
    \rho_{-}\ge O(N^{2+q-\tilde{\beta}}),\; \rho_{+}\ge O(N^{2+q-\tilde{\alpha}}).
\end{equation*}
In practice, we set $Q_{N}^{-} = D^2 P_N^- /N^{2+q-\tilde{\beta}},\; Q_{N}^{+} = D^2 P_N^+ /N^{2+q-\tilde{\alpha}}$ and tune the parameters $\rho_{\pm}$.
\end{Rem}

Now let us consider the case of FNBCs \eqref{IntroFNBC1.3}. In this case,
$\mathcal{B}_{-}u_{N}(-1)= \mathcal{D}_{p}^{\alpha-1}$ $u_{N}(-1)$, $\mathcal{B}_{+}u_{N}(1)= \mathcal{D}_{p}^{\alpha-1}u_{N}(1)$.
\begin{Theo}
Let $u_{N}$ be the solution of \eqref{Intro1.1} and \eqref{IntroFNBC1.3} and set $\omega(x) \equiv 1$. If
\begin{equation}\label{cond:coe:FNBC}
    \rho_{+}Q_N^{+} = \sum_{k=0}^{N}\frac{1}{\gamma_{k}^{\nu,\mu}}J_{k}^{-\nu,-\mu}(x)P_{k}^{\nu,\mu}(1),\;
     \rho_{-}Q_N^{-} = - \sum_{k=0}^{N}\frac{1}{\gamma_{k}^{\nu,\mu}}J_{k}^{-\nu,-\mu}(x)P_{k}^{\nu,\mu}(-1),
\end{equation}
where $\mu,\,\nu$ satisfying \eqref{PreJa2.13}, then
\begin{equation}\label{CoerFNBC}
\mathcal{A}_P^{R-L}(u_N,u_N,\rho_{\pm}, Q_N^{\pm})\geq \|u_{N}\|_{H_{RL}^{\alpha-1}(\Lambda)}^2,
\end{equation}
where the space $H_{RL}^{\alpha-1}(\Lambda)$ is given in \eqref{Frac:dspace}.
\end{Theo}
\begin{proof} Since $\omega(x) \equiv 1$, then using the integration by parts, we obtain
\begin{align*}
\mathcal{A}_P^{R-L}(u_N,u_N,\rho_{\pm}, Q_N^{\pm}) = &c(u_N,\mathcal{I}_{p}^{2-\alpha}u_{N}) + (\mathcal{D}_{p}^{\alpha-1}u_{N},\mathcal{D}_{p}^{\alpha-1}u_{N}) \\
 &+\mathcal{D}_{p}^{\alpha-1}u_{N}(-1)(\rho_{-}(Q_{N}^{-},\mathcal{I}_{p}^{2-\alpha}u_{N})+\mathcal{I}_{p}^{2-\alpha}u_{N}(-1))\\
& +\mathcal{D}_{p}^{\alpha-1}u_{N}(1)(\rho_{+}(Q_{N}^{+},\mathcal{I}_{p}^{2-\alpha}u_{N})-\mathcal{I}_{p}^{2-\alpha}u_{N}(1)).
\end{align*}
The condition \eqref{cond:coe:FNBC} yields
$\rho_{\pm}(Q_{N}^{\pm},\mathcal{I}_{p}^{2-\alpha}u_{N})\pm \mathcal{I}_{p}^{2-\alpha}u_{N}(\pm1)=0$.
Moreover, using the coercivity of the continuous problem, i.e., estimate \eqref{dbc:cont}, we obtain the estimate \eqref{CoerFNBC}.
\end{proof}
\begin{Rem}\label{rem4.2}
Same as in the case of FDBC, for  the condition \eqref{cond:coe:FDBC}, the penalty parameter $\rho_{+}$ (resp. $\rho_{-}$) and the function $Q_N^{+}$ (resp. $Q_N^{-}$) are also associated with each other.
To obtain the estimate of  $\rho_{\pm}Q_N^{\pm}$, we proceed  as follows:
\begin{align*}
         &\rho_{-}Q_N^{-}=\sum_{k=0}^{N}\frac{-1}{\gamma_{k}^{\nu,\mu}}J_{k}^{-\nu,-\mu}(x)P_{k}^{\nu,\mu}(-1) =-\omega^{\nu,\mu}(x)\tilde{h}_{N}^{\nu,\mu}P_{N}^{\nu,\mu+1}(x),\\
         &\rho_{+}Q_N^{+}=\sum_{k=0}^{N}\frac{1}{\gamma_{k}^{\nu,\mu}}J_{k}^{-\nu,-\mu}(x)P_{k}^{\nu,\mu}(1) =\omega^{\nu,\mu}(x)h_{N}^{\nu,\mu}P_{N}^{\nu+1,\mu}(x),
\end{align*}
where
\begin{equation*}
   \tilde{h}_{N}^{\nu,\mu}=\frac{ (-1)^{N}2^{-\nu-\mu-1} \Gamma(N+\nu+\mu+2)}{\Gamma(\mu+1)\Gamma(\nu+1+N)}, \quad
   h_{N}^{\nu,\mu}=\frac{2^{-\nu-\mu-1} \Gamma(N+\nu+\mu+2)}{\Gamma(\nu+1)\Gamma(\mu+1+N)}.
\end{equation*}
%
Using \eqref{CoerFDBC4.3}, we can get $h_{N}^{\nu,\mu}=O(N^{\nu+1})$ and
$\tilde{h}_{N}^{\nu,\mu}=O(N^{\mu+1})$.
Again, following the Theorem 3.24 in \cite{shen2011spectral}, we have $\max\limits_{x\in \Lambda}|P_{N}^{\nu+1,\mu}(x)|\sim N^{\nu+1} $ and
$\max\limits_{x\in \Lambda}|P_{N}^{\nu,\mu+1}|\sim N^{\mu+1}$. So
\begin{equation*}
\rho_{-}Q_N^{-}=\omega^{\nu,\mu}(x)\cdot O(N^{2\mu+2}), \quad \rho_{+}Q_N^{+}=\omega^{\nu,\mu}(x)\cdot O(N^{2\nu+2}).
\end{equation*}
Similarly,
we fix $Q_{N}^{\pm} = \omega^{\nu,\mu}(x)\cdot O(1)$ and let
     $\rho_{-}\ge O(N^{2\mu+2}),\; \rho_{+}\ge O(N^{2\nu+2})$
to obtain the coercivity (see the right plot of Figure \ref{fig:FNBC:Smoothf:eig} in Section \ref{sec:num}).
In practice, we set $Q_{N}^{-}(x)  = \frac{-1}{N^{2\mu+2}}\sum_{k=0}^{N}\frac{1}{\gamma_{k}^{\nu,\mu}}J_{k}^{-\nu,-\mu}(x)P_{k}^{\nu,\mu}(-1)$, $$Q_{N}^{+} = \frac{1}{N^{2\nu+2}} \sum_{k=0}^{N}\frac{1}{\gamma_{k}^{\nu,\mu}}J_{k}^{-\nu,-\mu}(x)P_{k}^{\nu,\mu}(1)$$
and tune the parameters $\rho_{\pm}$.
\end{Rem}
\subsection{SPM for the conservative  two-sided Caputo FDEs}\label{sec:spm:C}
In this subsection, we consider the two-sided Caputo FDEs with the Dirichlet BCs \eqref{IntroDBC1.4} or the Caputo FNBCs \eqref{IntroFNBC1.5}.
Unlike the case of R-L FDEs that looks for the solution in a \emph{poly-fractonomial} space, in this case, we seek the solution in the \emph{polynomial} space. The reason for this is that unlike the R-L FDEs where we use non-local Dirichlet BCs that require a poly-fractonomial basis so that they are bounded,  for the Caputo FDEs we use local Dirichlet BCs and hence no special treatment is needed.

\subsubsection{Numerical implementation of SPM}
The spectral penalty  scheme for \eqref{Intro1.1}-\eqref{IntroDBC1.4} or \eqref{Intro1.1}-\eqref{IntroFNBC1.5} is to find $u_N\in \mathbb{P}_N$, such that
\begin{equation}\label{CDBCPG}
 \mathcal{A}_P^{C}(u_N,v,\rho_{\pm}, Q_N^{\pm}) = \langle f,v\rangle_{\omega} + g_1 \rho_{-}(Q_N^{-},v)_{\omega} + g_2 \rho_{+}(Q_N^{+},v)_{\omega} \quad \forall \, v\in \mathbb{P}_N,
\end{equation}
where the bilinear form $\mathcal{A}_P^{C}(\cdot,\cdot,\cdot,\cdot)$ is given by
\begin{equation*}
  \begin{split}
  &\mathcal{A}_P^{C}(u,v,\rho_{\pm}, Q_N^{\pm})\\
  : =
  &  -(D\, \dxc^{\alpha-1}u,v)_{\omega}+c(u,v)_{\omega}+\rho_{-}\, \mathcal{B}_{-}u(-1)\,(Q_N^{-},v)_{\omega} +\rho_{+}\, \mathcal{B}_{+}u(1)\,(Q_N^{+},v)_{\omega},
  \end{split}
\end{equation*}
$\mathcal{B}_{\pm}u(\pm1)=u(\pm1)$ for the Dirichlet BCs \eqref{IntroDBC1.4} while $\mathcal{B}_{\pm}u(\pm1)=\dxc^{\alpha-1}u(\pm1)$ for the Caputo fractional BCs \eqref{IntroFNBC1.5}. Again, $Q_N^{\pm}$ and $\rho_{\pm}$ are penalty functions and parameters  to be determined.

Taking $u_N(x)=\sum_{k=0}^{N}\tilde{u}_{k}L_{k}(x)$ where $L_{k}(x)=\j_k^{0,0}(x),k\ge 0$ are the Legendre polynomials,
and letting the test functions be $L_i(x), 0\leq i\leq N$, gives the linear system
\begin{equation}\label{CLMS}
  (-\mathcal{S} + c\mathcal{M} +\mathcal{B})U=\widehat{\mathcal{F}}+\widetilde{\mathcal{F}},
\end{equation}
where $U=(\tilde{u}_0,\tilde{u}_1,\cdots,\tilde{u}_{N})^T$ and
\begin{equation*}
\begin{aligned}
  &\mathcal{S}=(s_{ik})_{i,k=0}^{N}, \; s_{ik}=(D\,\dxc^{\alpha}L_k,L_i)_{\omega};
  \quad \mathcal{M}=(m_{ik})_{i,k=0}^{N}, \;  m_{ik}=(L_k,L_{i})_{\omega};\\
  &\mathcal{B}=(b_{ik}^{+} + b_{ik}^{-})_{i,k=0}^{N}, \;  b_{ik}^{\pm}=\rho_{\pm}(\mathcal{B}_{\pm}L_k)(\pm 1) (Q_N^{\pm},L_{i})_{\omega};\\
  & \widehat{\mathcal{F}}=(\hat{f}_{0},\cdots,\hat{f}_{N})^{T},\; \hat{f}_{i}=\langle f,L_{i}\rangle_{\omega}; \\
  & \widetilde{\mathcal{F}}=(\tilde{f}_0,\cdots,\tilde{f}_{N})^{T},\; \tilde{f}_{i}= \rho_{+}g_2(Q_N^{+},L_{i})_{\omega} +\rho_{-}g_1(Q_N^{-},L_{i})_{\omega}.
\end{aligned}
\end{equation*}

Next, we briefly show how to compute the stiffness matrix $\mathcal{S}$. We begin by presenting the following result:
\begin{Lem}\label{lem5.1}\cite{MaoShen16}
For $1<\alpha<2$, we have
\begin{equation}\label{CFDE1}
\begin{aligned}
  &_{-1}D_{x}^{\alpha-1}{L_n}(x)=r_{n\a}(1+x)^{1-\alpha}P_n^{(\alpha-1,1-\alpha)}(x),\\
  &  _{x}D_{1}^{\alpha-1}L_n(x)=r_{n\a}(1-x)^{1-\alpha}P_n^{(1-\alpha,\alpha-1)}(x)
\end{aligned}
\end{equation}
and
\begin{equation}\label{CFDE2}
\begin{aligned}
  &_{-1}I_{x}^{2-\alpha}L_n(x)=z_{n\a}(1+x)^{2-\alpha}P_n^{(\alpha-2,2-\alpha)}(x),\\
  &_{x}I_{1}^{2-\alpha}L_n(x)=z_{n\a}(1-x)^{2-\alpha}P_n^{(2-\alpha,\alpha-2)}(x),
\end{aligned}
\end{equation}
%
where $r_{n\a}=\frac{\Gamma(n+1)}{\Gamma(n+2-\alpha)}$ and $z_{n\a}=\frac{\Gamma(n+1)}{\Gamma(n+3-\alpha)}$.
\end{Lem}

To compute $s_{ik},i,k = 0,\ldots, N$, we need to compute both $(D\,\cdlx^{\a-1}L_k,L_i)_{\omega}$ and $(D\,\cdrx^{\a-1}L_k,L_i)_{\omega}$. By using the definitions of the Caputo and R-L fractional derivatives,  equations \eqref{CFDE1} and \cite[(3.176b)]{shen2011spectral}, we compute $D\,\cdlx^{\a-1}L_k$ as follows:
\begin{equation*}
\begin{aligned}
  D\,\cdlx^{\a-1}L_k(x) &= \dlx^{\a-1} L'_k(x) = \dlx^{\a-1}\sum_{n=0, n+k~ odd }^{k-1}(2n+1)L_{n}(x) \\
  &= (1+x)^{1-\alpha} \sum_{n=0,n+k~ odd }^{k-1}(2n+1) r_{n,\a} P_n^{(\alpha-1,1-\alpha)}(x).
\end{aligned}
\end{equation*}
Then the inner product $(D\,\cdlx^{\a-1}L_k,L_i)_{\omega}$ can be computed by using Gauss quadrature with prescribed weight function $\omega(x)$. The same procedure can be applied for $(D\,\cdrx^{\a-1}L_k,L_i)_{\omega}$, and then we can obtain the element $s_{ik}$.
Furthermore, this procedure can also be used to compute  the boundary matrix $\mathcal{B}$ for the Caputo FNBCs.

\subsubsection{A sufficient condition for the coercivity of  \eqref{CDBCPG} with Caputo FNBCs}
For the coercivity of the spectral penalty scheme \eqref{CDBCPG} with Dirichlet BCs \eqref{IntroFNBC1.3}, we cannot provide a rigorous analysis due to technical  difficulties. Next, we only show the  analysis of coercivity in the case of Caputo FNBCs.
\begin{Theo}\label{thm:Caputo:FNBC}
Let $\omega(x) \equiv 1$ and $u_{N}$ be the solution of \eqref{CDBCPG} with Caputo FNBCs. If
\begin{equation}\label{cond:coe:FNBC:C}
    \rho_{-}Q_N^{-}(x)=-\sum_{k=0}^{N}\frac{1}{\gamma_{k}^{0,0}}L_{k}(x)L_{k}(-1),\;
    \rho_{+}Q_N^{+}(x)=\sum_{k=0}^{N}\frac{1}{\gamma_{k}^{0,0}}L_{k}(x)L_{k}(1),
\end{equation}
then
\begin{equation}\label{CoerFNBC:C}
\mathcal{A}_P^{C}(u_N,u_N,\rho_{\pm}, Q_N^{\pm})\geq \|u'_{N}\|_{J^{\a/2-1}(\Lambda)}^2.
\end{equation}
\end{Theo}
\begin{proof} Since $\omega(x) \equiv 1$, then using the integration by parts, we obtain
\begin{align*}
 \mathcal{A}_P^{C}(u_N,u_N,\rho_{\pm}, Q_N^{\pm})
 =&c(u_N,u_{N})+(\mathcal{I}_{p}^{2-\alpha}Du_N,Du_{N})\\
 &+\dxc^{\alpha-1}u_{N}(-1)(\rho_{-}(Q_{N}^{-},u_{N})+u_{N}(-1))\\
 &+\dxc^{\alpha-1}u_{N}(1)(\rho_{+}(Q_{N}^{+},u_{N})-u_{N}(1)).
\end{align*}
Using the condition \eqref{cond:coe:FNBC:C} gives
$\rho_{\pm}(Q_{N}^{\pm},u_{N})\pm u_{N}(\pm1)=0$.
We then obtain the estimate \eqref{CoerFNBC:C} from the estimate \eqref{eqn:equv} and \eqref{eqn:fintpart}.
\end{proof}
\begin{Rem}\label{rem5.1}
Again,
using the same argument, we can obtain
\begin{align*}
\rho_{-}Q_N^{+}=O(N^{2}), \quad \rho_{+}Q_N^{-}=O(N^{2}).
\end{align*}
Furthermore, we fix $Q_{N}^{\pm} = O(1)$  and let
    $\rho_{-}\ge O(N^{2}),\; \rho_{+}\ge O(N^{2})$
%
to obtain the coercivity (see the right plot Figure \ref{fig:CaFNBC:Smoothf:eig} in Section \ref{sec:num}).
In practice, we set $Q_{N}^{\pm}(x) = \pm\frac{1}{N^{2}}\sum_{k=0}^{N}\frac{1}{\gamma_{k}^{0,0}}L_{k}(x)L_{k}(\pm1)$  and tune the parameters $\rho_{\pm}$.
\end{Rem}

\begin{Rem}\label{rem4.4}
Although we are unable to give a rigorous analysis for the case of classical Dirichlet BCs, we can provide an intuitive approach on how to choose the penalty parameters and associated functions.  Specifically, let $\omega(x) \equiv 1$ and
\begin{equation}\label{eqn:caputo:QN:DBC}
    Q_N^{\pm}(x) = \frac{1}{N^{2}}\sum_{k=0}^{N}\frac{1}{\gamma_{k}^{0,0}}L_{k}(x)L_{k}(\pm1),
\end{equation}
we have  $Q_N^{\pm}(x) = O(1)$ and
  $(Q_{N}^{\pm},u_{N})=u_{N}(\pm 1)/N^2$.
Thus,
\begin{align*}
 &\mathcal{A}_P^{C}(u_N,u_N,\rho_{\pm}, Q_N^{\pm})\\
 =&c(u_N,u_{N})-(D\mathcal{I}_{p}^{2-\alpha}Du_N,u_{N})+u_{N}(-1)\rho_{-}(Q_{N}^{-},u_{N})+u_{N}(1)\rho_{+}(Q_{N}^{+},u_{N})\\
 =&c(u_N,u_{N})-(D\mathcal{I}_{p}^{2-\alpha}Du_N,u_{N})+\rho_{-}u^{2}_{N}(-1)/N^2+\rho_{+}u^{2}_{N}(1)/N^2.
\end{align*}
To ensure  coercivity, we can provide sufficient large values of $\rho_{\pm}$ such that
$$\mathcal{A}_P^{C}(u_N,u_N,\rho_{\pm}, Q_N^{\pm}) \ge C>0.$$
The numerical results show that
\begin{align*}
\rho_{+} = \rho_{-}=O(N^{3})
\end{align*}
is a good choice, see the results for Example \ref{ex:CaDBC}.
%
%
%
\end{Rem}

\section{Numerical examples}\label{sec:num}
We now show several numerical examples to illustrate the accuracy and coercivity conditions of the proposed SPM, and we will compare the results against results obtained from PGS-$\tau$.

\subsection{Numerical tests for the conservative R-L FDEs}
\begin{exam}\label{ex:RLFDBC}
We begin by considering the conservative R-L FDEs with FDBCs, i.e., \eqref{Intro1.1}-\eqref{IntroFDBC1.2}. In particular, we consider the following two cases:
\begin{itemize}
  \item Case I: Smooth solution $u(x) = (1-x^2)^2$;
  \item Case II: Smooth right hand function (RHF) $f(x) = 1+\cos(\pi x)$.
\end{itemize}
For Case I, the boundary conditions can be computed directly by the exact solution while for Case II, the boundary conditions are $\ix^{2-\a}u(-1)=2,\;  \ix^{2-\a}u(1)=1$.  
\end{exam}

We first test the accuracy to illustrate the effectiveness of SPM \eqref{RLPG}. Set $p =0.8$ and $\tilde{\a}=\frac{\a}{2},\; \tilde{\b}=\frac{\a}{2}$.
According to Remark \ref{rem4.1}, we
take $\rho_{-}=N^{2+q-\tilde{\beta}},\; \rho_{+}=N^{2+q-\tilde{\a}}$ with $q=max(\tilde{\a}, \tilde{\b}) +2$, and  $Q_{N}^{-}(x) = D^2 P_N^-(x) /N^{2+q-\tilde{\beta}},\; Q_{N}^{+} = D^2 P_N^+(x) /N^{2+q-\tilde{\alpha}}$, where $P_N^{\pm}(x)$ are given in \eqref{eqn:PNpm}.
We compute the $L^{\infty}$ error using 1000 uniformly distributed points. For comparison, we also compute the  $L^{\infty}$ error using PGS-$\tau$ presented in Appendix \ref{sec:app:tau}.
The convergence results of the $L^{\infty}$ error with different values of fractional order $\alpha = 1.2,1.8$ and $c = 0,1$ for Case I and Case II are shown in Figure \ref{fig:FDBC:Smoothu:errcomp} and \ref{fig:FDBC:Smoothf:errcomp}, respectively. Here, all the parameters for the PGS-$\tau$  are the same as the ones for SPM except the penalty parameters. For Case II, i.e., for the case of smooth RHF, since we do not have the analytic solution, we obtain the numerical solution using SPM with $N = 512$ as the reference solution; the same approach is also used for all the tests below, which require a reference solution but do not have an explicit one.
Observe from both Figures that the accuracy with SPM is \emph{higher} than that with PGS-$\tau$ for \emph{all} cases.
For Case I,  algebraic convergence is obtained; see Figure \ref{fig:FDBC:Smoothu:errcomp} (left) for $c = 0$.
For case II, Figure \ref{fig:FDBC:Smoothf:errcomp} shows that spectral convergence is obtained for $c = 0$ while algebraic convergence is obtained for $c = 1$. This means that for $c = 0$, the convergence of SPM depends only on the regularity of the RHF.
We also show how the penalty parameters $\rho_{\pm}$ affect the accuracy of SPM. The right plot of Figure \ref{fig:FDBC:Smoothu:errcomp} shows the $L^{\infty}$-error against the penalty parameters $\rho_{+}$ with different values of fractional order $\alpha = 1.2,\, 1.8$ for $c = 0$; similar results are obtained for $c = 1$ not shown here. We observe that we obtain the \emph{best} accuracy  when the penalty parameters are chosen to  satisfy the condition \eqref{cond:coe:FDBC}, even for the subcase of $c = 1$ (not shown here) that is not covered by our theory.

In order to verify the sufficient condition \eqref{cond:coe:FDBC} for the coercivity, we calculate the minimum value of the real part of all eigenvalues, denoted by $Re(eig)_{min}$, for different values of fractional order with $N = 100$ and $p = 0.8,\,0.5$, which are shown in Figure \ref{fig:minieig:allalpha}. The left plot shows the results for different values of fractional order while the right plot shows the results against the penalty parameter $\rho_+$, both for $c = 0$.
Observe that the values of $Re(eig)_{min}$ for $\alpha\in(1,2)$ are positive, which means that SPM \eqref{RLPG} with FDBCs is coercive provided that the condition \eqref{cond:coe:FDBC} holds; this  agrees with our analysis. For $c = 1$, similar results are obtained  (not shown here) although this case is not covered by our theory.
Overall, we observe positivity
of values of $Re(eig)_{min}$, for $c = 0,1$ provided that
\begin{equation*}
    \rho_{-}\ge N^{2+q-\tilde{\beta}},\; \rho_{+}\ge N^{2+q-\tilde{\alpha}}
\end{equation*}
as discussed in Remark \ref{rem4.1}.

\begin{figure}[!t]
\begin{minipage}{0.49\linewidth}
\begin{center}
\includegraphics[scale=0.4,angle=0]{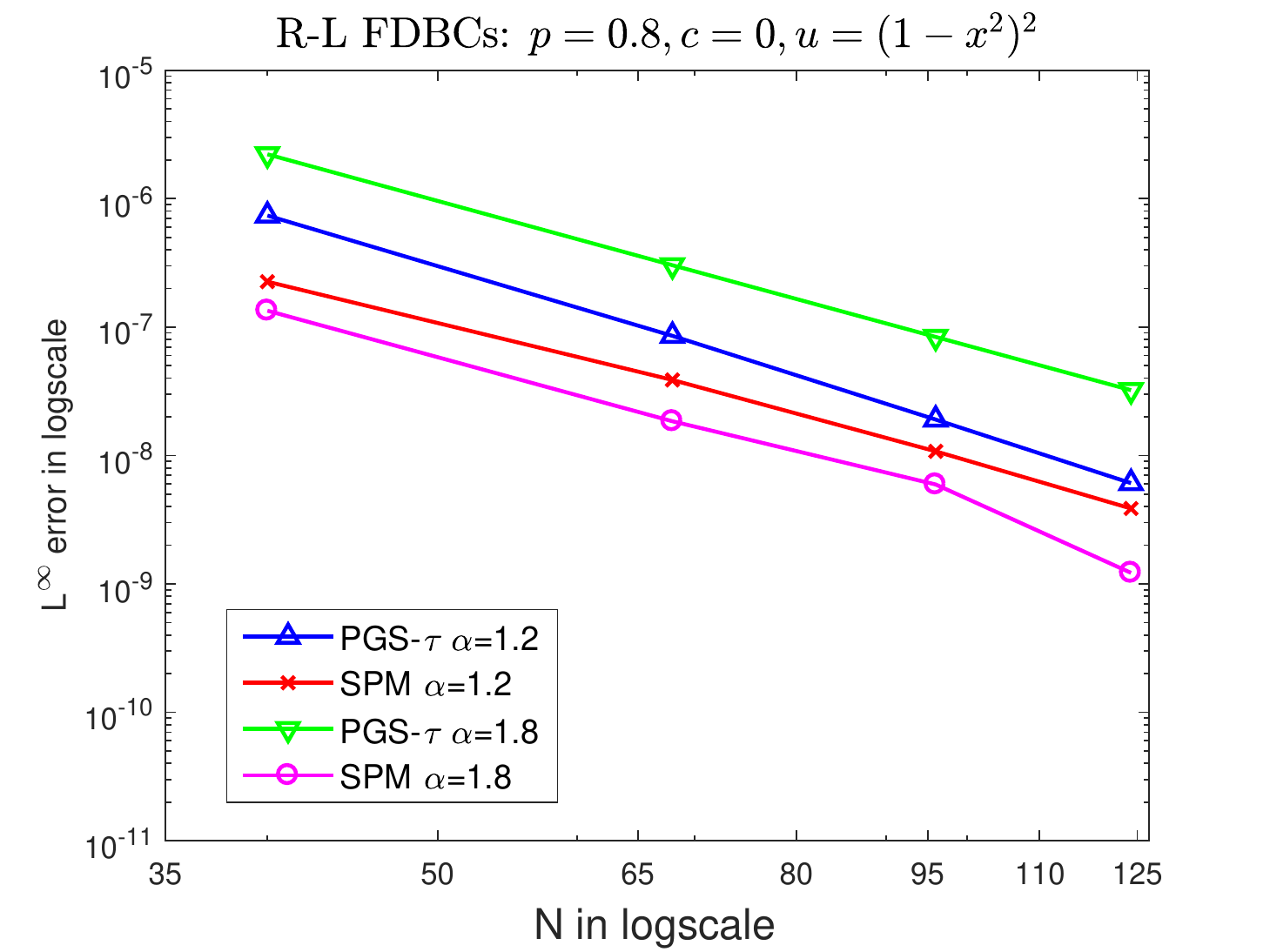}\end{center}
\end{minipage}
\begin{minipage}{0.49\linewidth}
\begin{center}
\includegraphics[scale=0.4,angle=0]{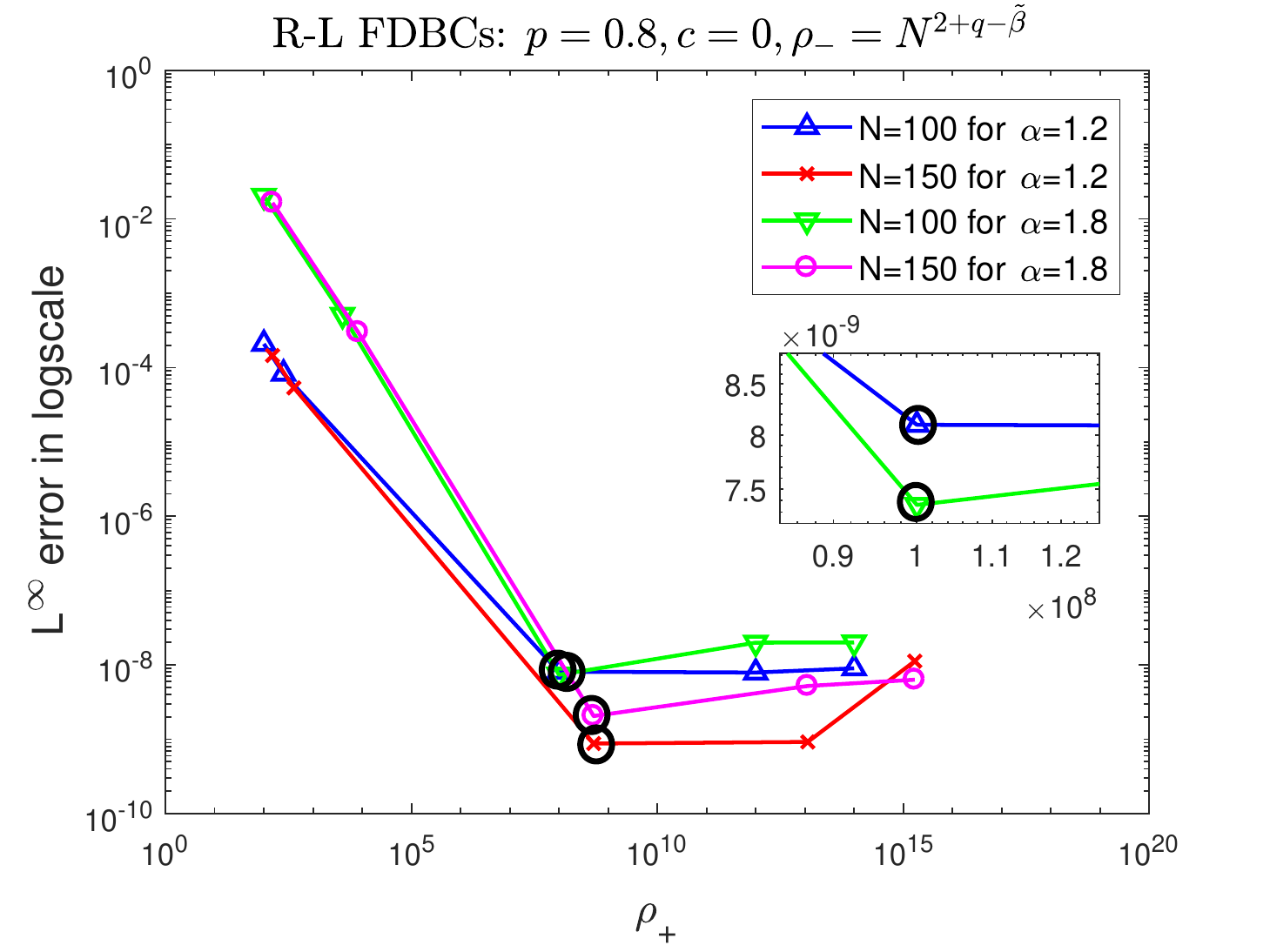}\end{center}
\end{minipage}
\caption{\scriptsize
Case I of Example \ref{ex:RLFDBC}: Convergence of $L^{\infty}$-error for SPM and PGS-$\tau$ versus polynomial order $N$ (left) and versus penalty parameter $\rho_+$ (right) for different values of fractional order $\alpha=1.2,\; 1.8$. The black circles (right) correspond to the penalty parameters satisfying the coercivity sufficient condition \eqref{cond:coe:FDBC}.
}\label{fig:FDBC:Smoothu:errcomp}
\end{figure}


\begin{figure}[!t]
\begin{minipage}{0.49\linewidth}
\begin{center}
\includegraphics[scale=0.4,angle=0]{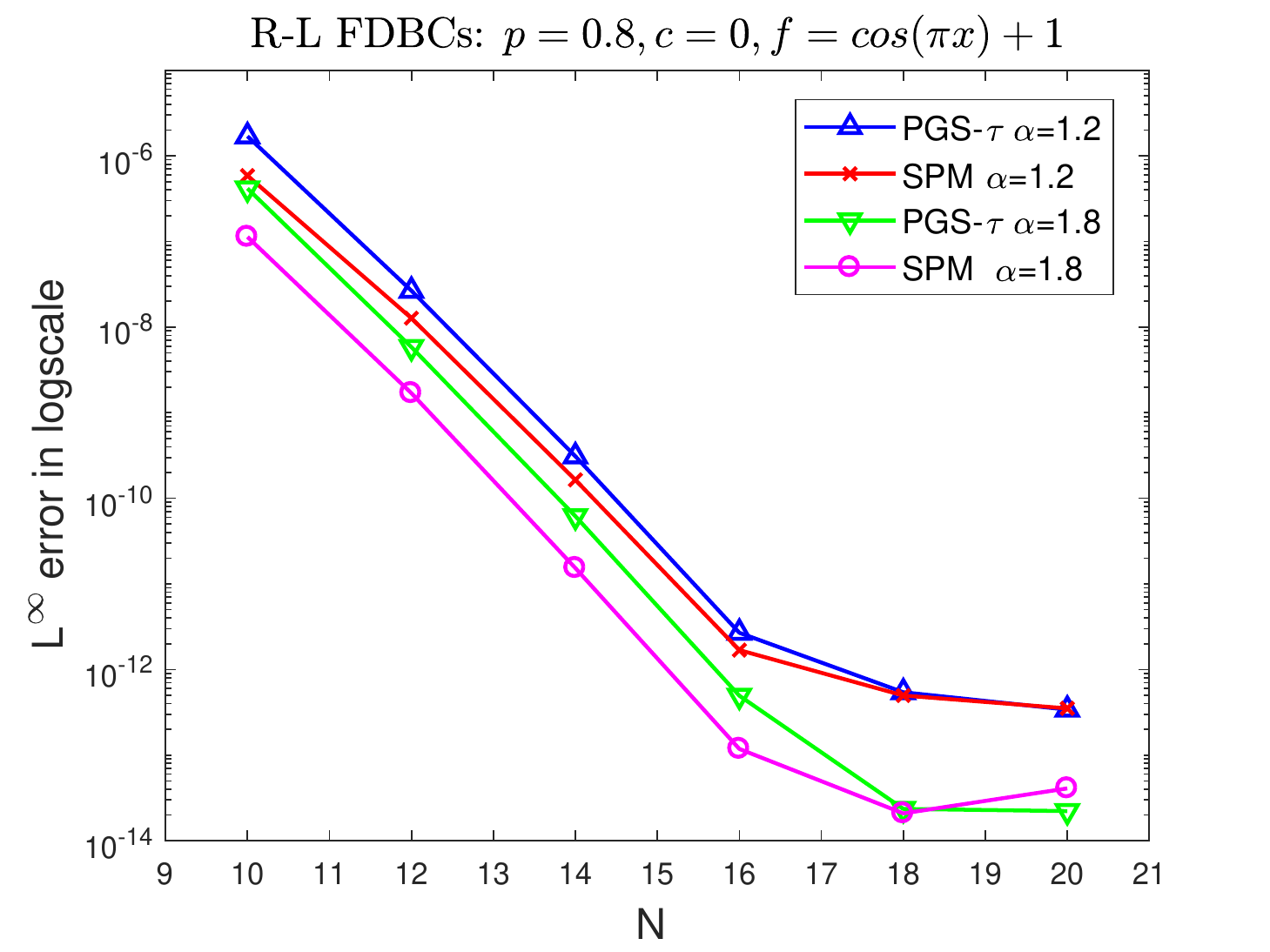}\end{center}
\end{minipage}
\begin{minipage}{0.49\linewidth}
\begin{center}
\includegraphics[scale=0.4,angle=0]{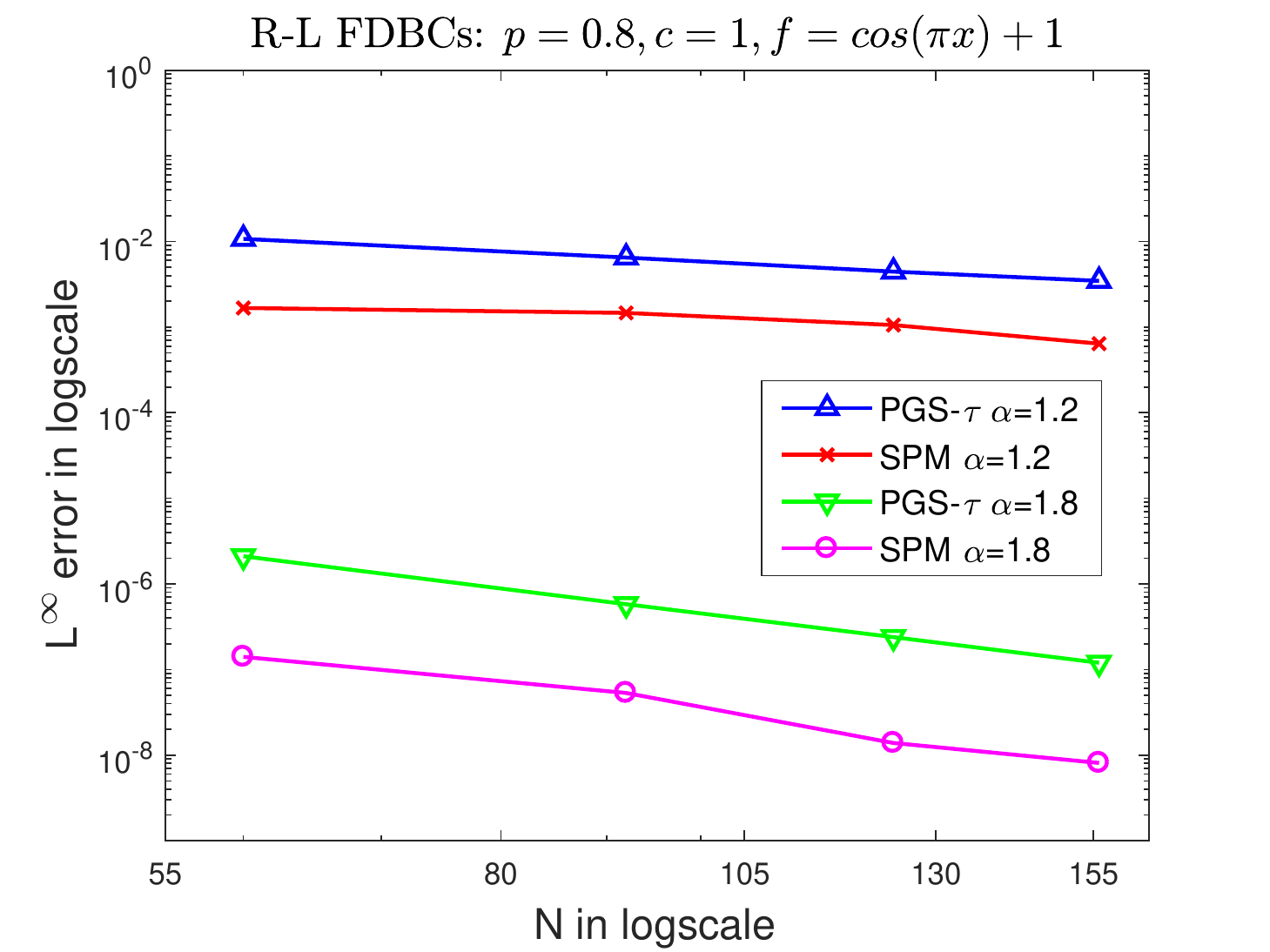}\end{center}
\end{minipage}
\caption{\scriptsize
Case II of Example \ref{ex:RLFDBC}: Convergence of $L^{\infty}$-error for SPM and PGS-$\tau$ with different values of fractional order $\alpha=1.2,\; 1.8$ and $\rho_{-}=N^{2+q-\tilde{\b}},\,\rho_{+}=N^{2+q-\tilde{\a}}$. Left: $c=0$, right: $c=1$.}\label{fig:FDBC:Smoothf:errcomp}
\end{figure}


\begin{figure}[!t]
\begin{minipage}{0.49\linewidth}
\begin{center}
\includegraphics[scale=0.4,angle=0]{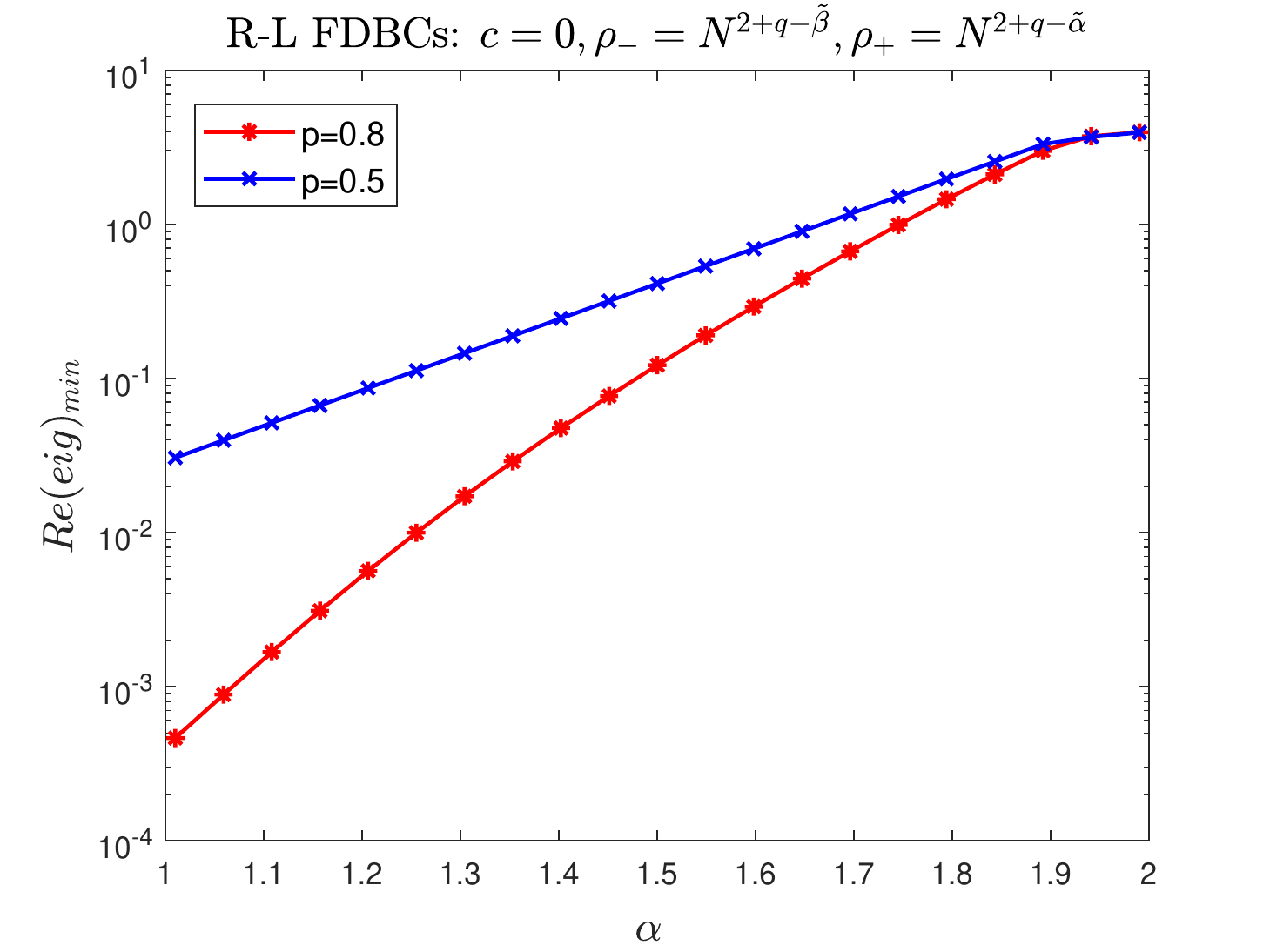}\end{center}
\end{minipage}
\begin{minipage}{0.49\linewidth}
\begin{center}
\includegraphics[scale=0.4,angle=0]{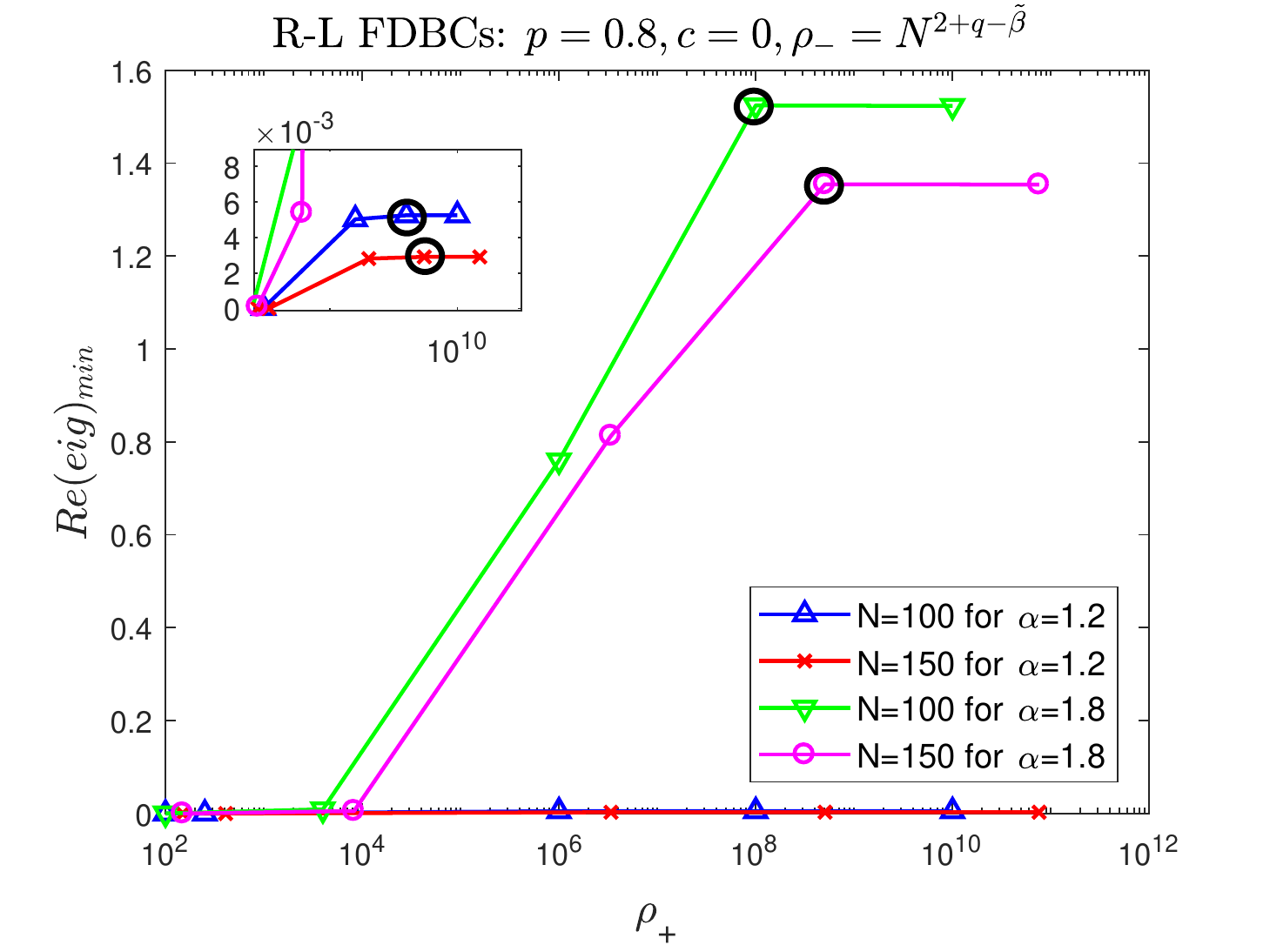}\end{center}
\end{minipage}
\caption{\scriptsize
Example \ref{ex:RLFDBC}: The minimum values of the real part  of eigenvalues versus fractional order (left) and penalty parameter $\rho_+$ (right). The black circles (right) correspond to the penalty parameters satisfying the coercivity sufficient condition \eqref{cond:coe:FDBC}.}\label{fig:minieig:allalpha}
\end{figure}

\begin{exam}\label{ex:RLFNBC}
We now consider the conservative R-L FDEs with FNBCs, i.e., \eqref{Intro1.1}-\eqref{IntroFNBC1.3}. We consider the following two cases as in the previous Example:
\begin{itemize}
  \item Case I: Smooth solution $u(x) = (1-x^2)^2$;
  \item Case II: Smooth RHF $f(x) = 1+\cos(\pi x)$.
\end{itemize}
For Case I, the boundary conditions can be computed directly by the exact solution while for Case II, the boundary conditions are $\dx^{\a-1}u(-1)=2,\;  \dx^{\a-1}u(1)=1$.  
\end{exam}

In this example, we fix $c = 1$. For $c = 0$, we require an additional condition of mass conservation, but we will not discuss this here.
By  Remark \ref{rem4.2}, we set  $Q_{N}^{-}(x)  = \frac{-1}{N^{2\mu+2}}\sum_{k=0}^{N}\frac{1}{\gamma_{k}^{\nu,\mu}}J_{k}^{-\nu,-\mu}(x)P_{k}^{\nu,\mu}(-1)$, $Q_{N}^{+} = \frac{1}{N^{2\nu+2}} \sum_{k=0}^{N}\frac{1}{\gamma_{k}^{\nu,\mu}}J_{k}^{-\nu,-\mu}(x)P_{k}^{\nu,\mu}(1)$
with $\mu,\nu$ satisfying  condition \eqref{PreJa2.13}.
The convergence results of the $L^{\infty}$-error for SPM and PGS-$\tau$ with different values of fractional order $\alpha=1.2,\, 1.8$ with $p = 0.8,\, \rho_{-} = N^{2\mu+2},\, \rho_{+}= N^{2\nu+2}$ are shown for Case I and II in the upper left and right plots of Figure \ref{fig:FNBC:Smoothf:errcomp}, respectively. Again, in both cases, we obtain \emph{higher} accuracy with SPM than with  PGS-$\tau$.
We also show $L^{\infty}$-error for different fractional orders $\alpha=1.2,\; 1.8$ by tuning the penalty parameters $\rho_{+}$ in the lower plot of Figure \ref{fig:FNBC:Smoothf:errcomp}.  We observe again that the \emph{best} accuracy is obtained when we choose $\rho_{-} = N^{2\mu+2},\, \rho_{+}= N^{2\nu+2}$ for which the sufficient condition \eqref{cond:coe:FNBC} of coercivity is satisfied.
To verify the coercivity condition \eqref{cond:coe:FNBC} of SPM \eqref{RLPG} with R-L FNBCs, we show the values of $Re(eig)_{min}$ for $\alpha\in (1,2)$ with $p = 0.8,\,0.5$ in the left plot of Figure \ref{fig:FNBC:Smoothf:eig}. We can see that all  values of $Re(eig)_{min}$ are positive. This verifies the coercivity condition \eqref{cond:coe:FNBC}. Moreover, we can see from the right plot of Figure \ref{fig:FNBC:Smoothf:eig} that coercivity can be maintained if $\rho_{-} \ge N^{2\mu+2},\, \rho_{+} \ge N^{2\nu+2}$.

\begin{figure}[!t]
\begin{minipage}{0.49\linewidth}
\begin{center}
\includegraphics[scale=0.4,angle=0]{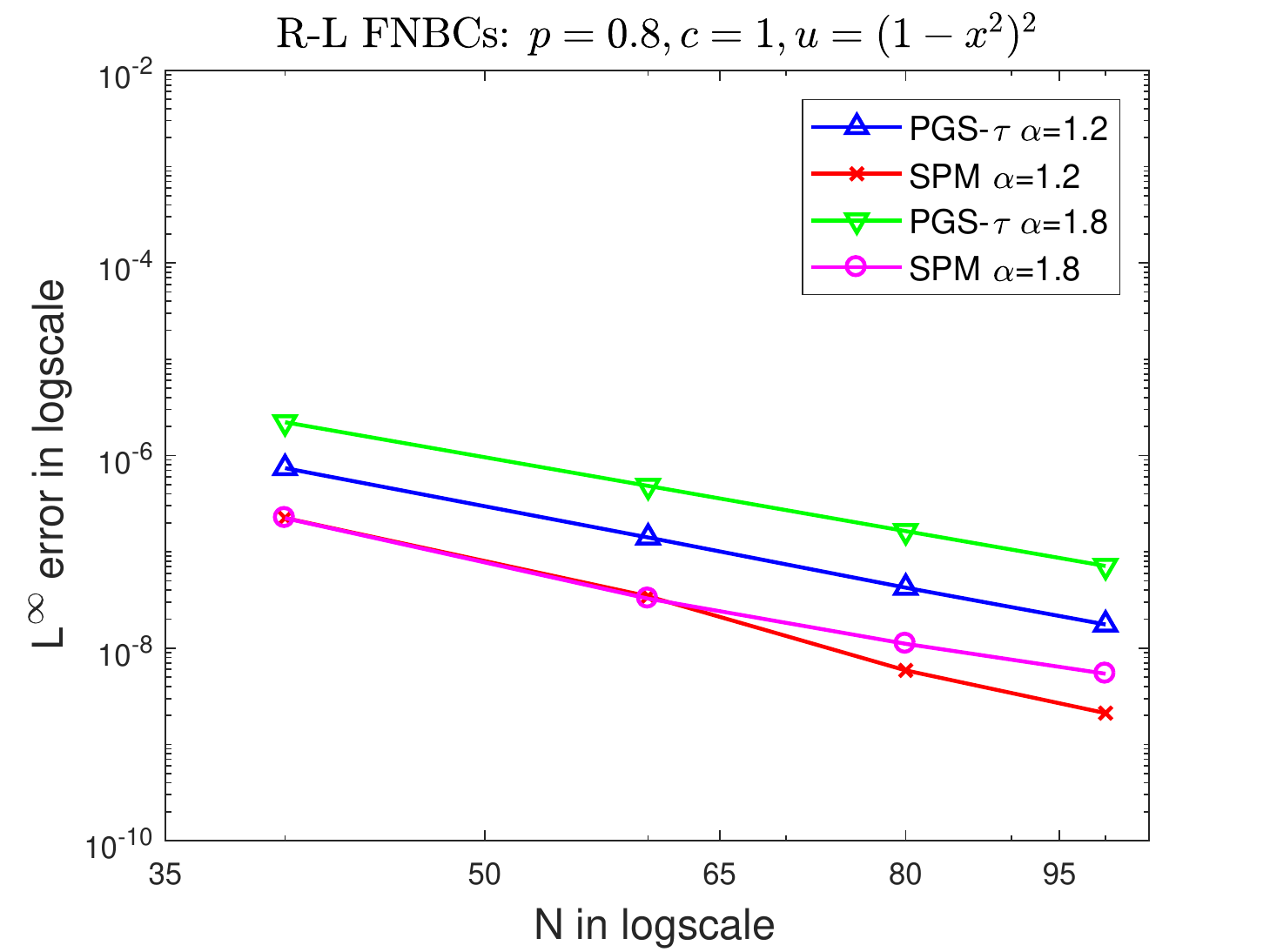}\end{center}
\end{minipage}
\begin{minipage}{0.49\linewidth}
\begin{center}
\includegraphics[scale=0.4,angle=0]{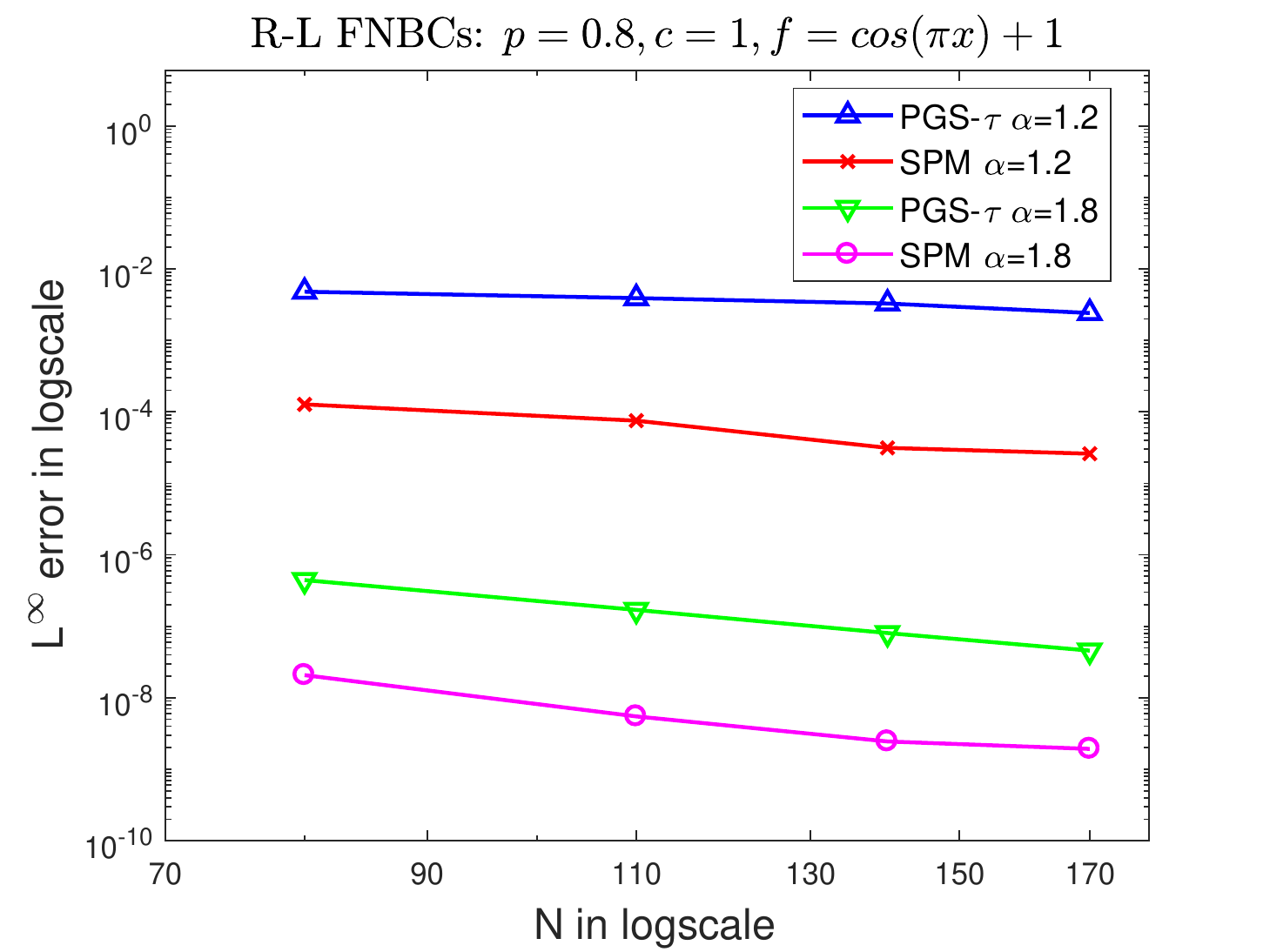}\end{center}
\end{minipage}
\begin{minipage}{0.95\linewidth}
\begin{center}
\includegraphics[scale=0.4,angle=0]{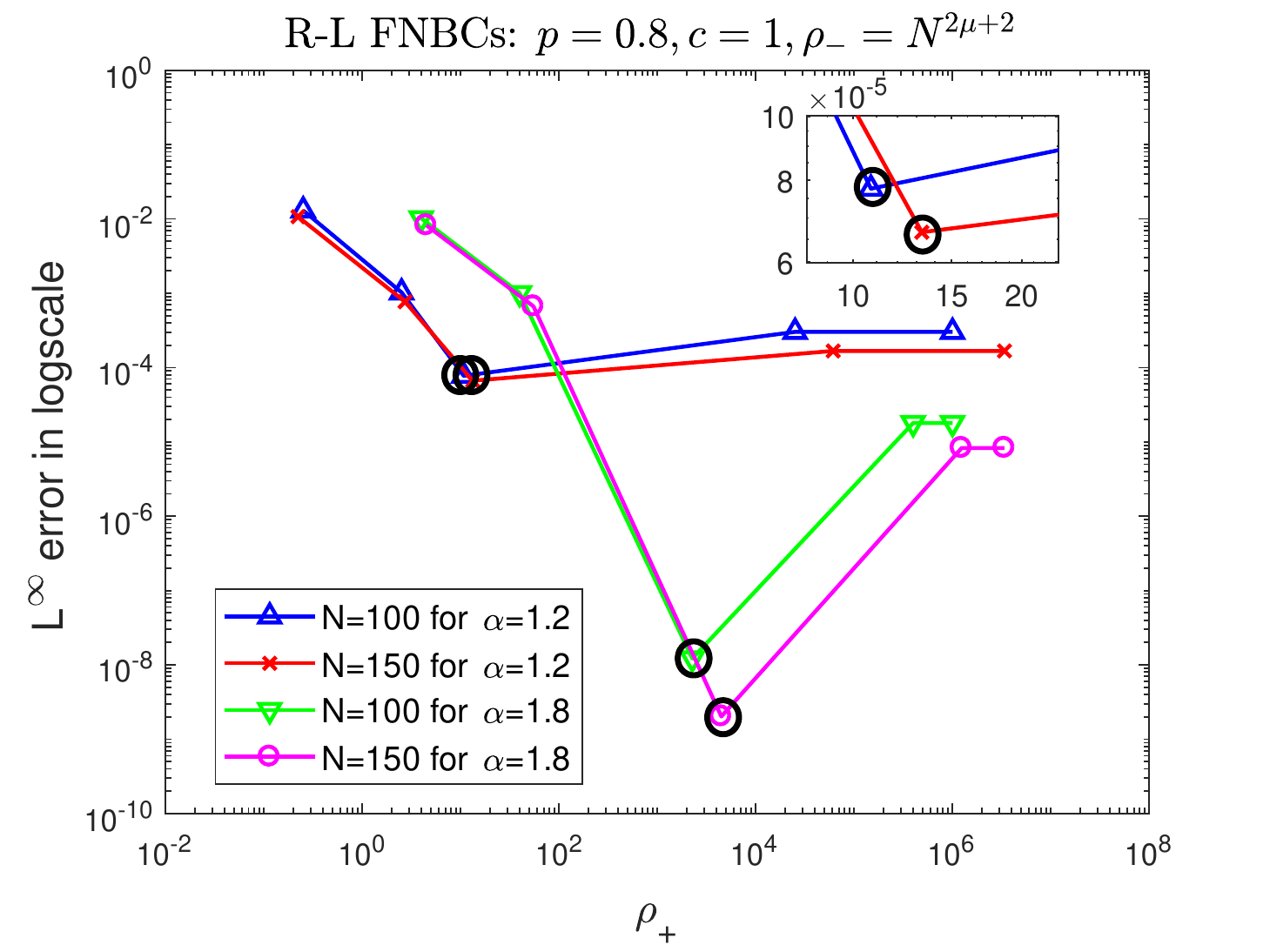}\end{center}
\end{minipage}
\caption{\scriptsize
Example \ref{ex:RLFNBC}:  Convergence of $L^{\infty}$-error for SPM and PGS-$\tau$ with different values of fractional order $\alpha = 1.2,1.8$. Upper left: versus $N$ for smooth solution $u(x) = (1-x^2)^2$, upper right: versus $N$ for smooth RHF $f(x) = 1+\cos(\pi x)$, lower: versus the penalty parameter $\rho_+$. The black circles (down) correspond to the penalty parameters satisfying the coercivity sufficient condition \eqref{cond:coe:FNBC}.}\label{fig:FNBC:Smoothf:errcomp}
\end{figure}



\begin{figure}[!t]
\begin{minipage}{0.49\linewidth}
\begin{center}
\includegraphics[scale=0.4,angle=0]{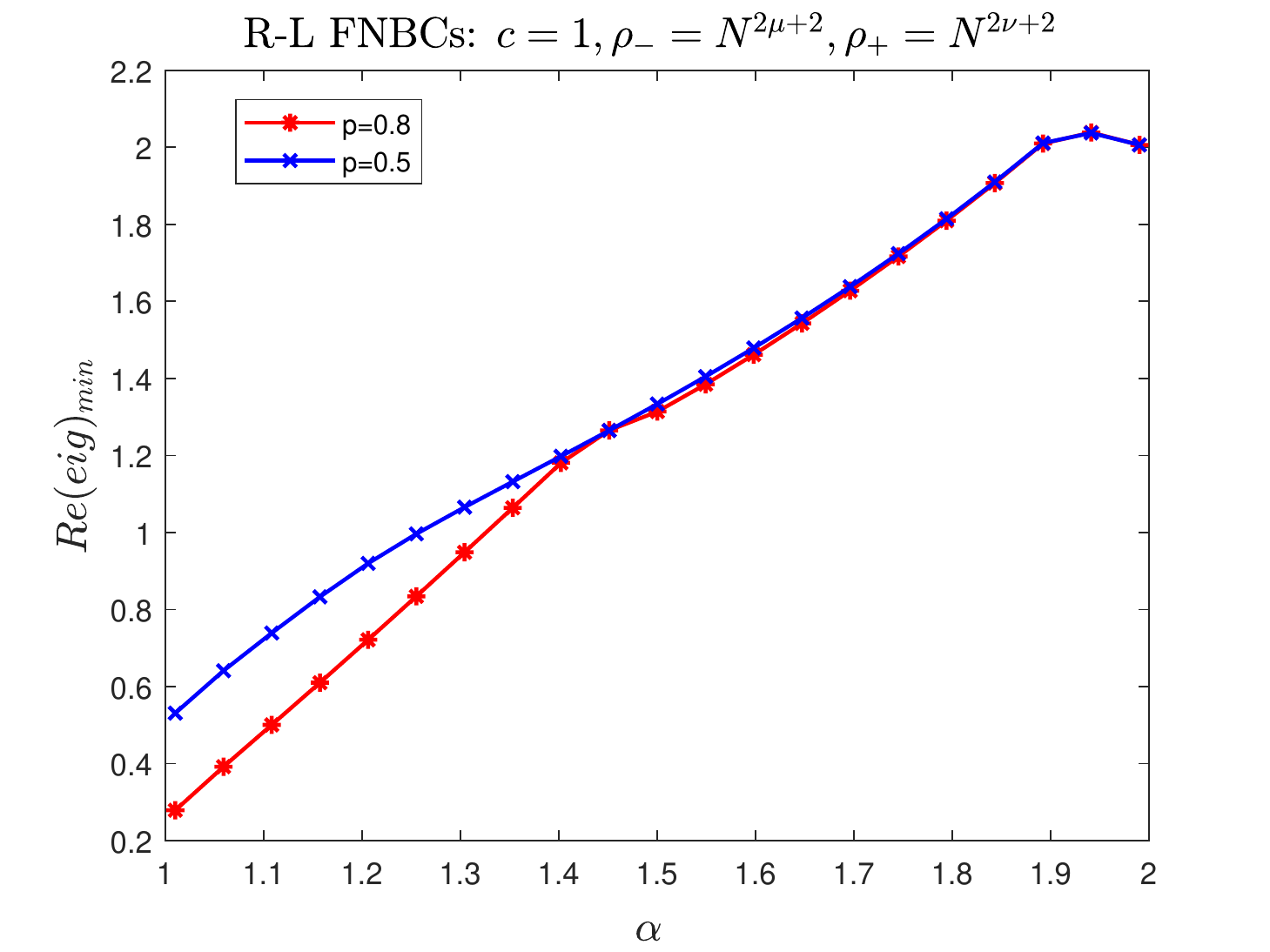}\end{center}
\end{minipage}
\begin{minipage}{0.49\linewidth}
\begin{center}
\includegraphics[scale=0.4,angle=0]{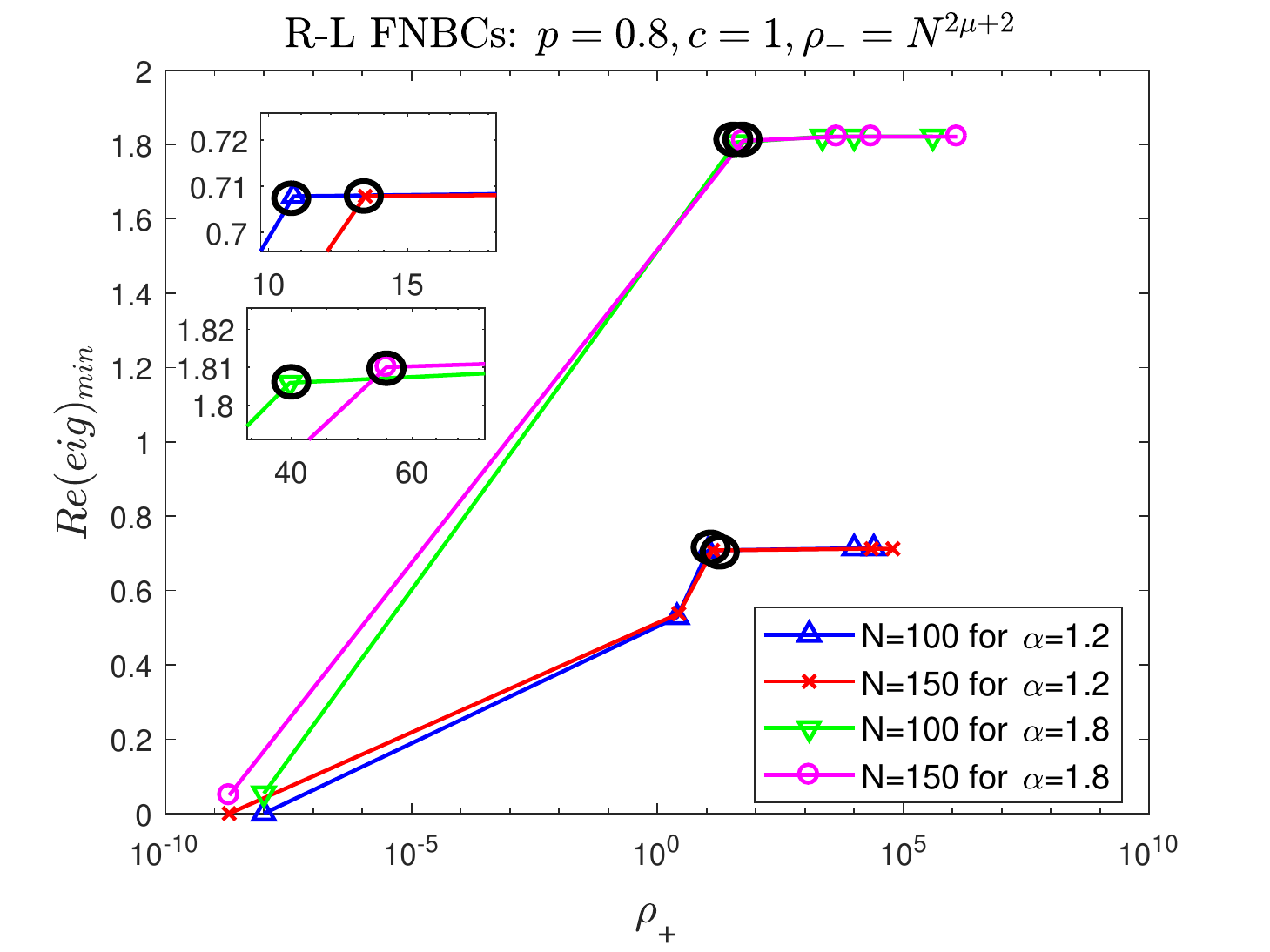}\end{center}
\end{minipage}
\caption{\scriptsize
Example \ref{ex:RLFNBC}: The minimum values of the real part  of eigenvalues versus fractional order (left) and penalty parameter $\rho_+$ (right). The black circles (right) correspond to the penalty parameters satisfying the coercivity sufficient condition \eqref{cond:coe:FNBC}.}\label{fig:FNBC:Smoothf:eig}
\end{figure}


\subsection{Numerical tests for the conservative Caputo FDEs}\label{sec:num:Caputo:DBC}
\begin{exam}\label{ex:CaDBC}
We now turn to the Caputo fractional FDEs. Consider the conservative Caputo FDEs with classical Dirichlet BCs, i.e., \eqref{Intro1.1}-\eqref{IntroDBC1.4}, with the following two cases:
\begin{itemize}
  \item Case I: Smooth solution $u(x) = \cos(\pi x)$;
  \item Case II: Smooth RHF $f(x) = 1+\cos(\pi x)$.
\end{itemize}
For Case I, the boundary conditions can be computed directly by the exact solution while for Case II, the boundary conditions are $u(-1)=1,\;  u(1)=2$.  
\end{exam}

By the virtue of the discussion of Remark \ref{rem4.4},
we take the penalty functions  given by \eqref{eqn:caputo:QN:DBC}, i.e.,
    $Q_N^{\pm}(x) = \frac{1}{N^2}\sum_{k=0}^{N}\frac{1}{\gamma_{k}^{0,0}}L_{k}(x)L_{k}(\pm1)$.
We now test the accuracy by choosing the parameters to be $\rho_+ = \rho_- = N^3$.
Figure \ref{fig:CaDBC:Smoothuf:errcomp:c0} shows the convergence results  for the Case I (upper left) and Case II (upper right) with $p = 0.8, c = 0$ and different values of fractional order $\alpha = 1.2,\, 1.8$.
Observe that we can obtain \emph{spectral accuracy} for the smooth solution, which is expected since we use the polynomial approximation, while algebraic convergence is obtained for the case of smooth RHF. However, for the case of smooth RHF, we again observe that the accuracy with SPM is much \emph{higher} than that with PGS-$\tau$.
Next, we present the $L^{\infty}$-error with respect to the values of the penalty parameters $\rho_{\pm}$.
For the sake of simplicity, we let $\rho_{-} = \rho_{+} = \rho$.
By tuning the parameters $\rho_{\pm}$,
we plot the  $L^{\infty}$-error against values of the penalty parameter $\rho = \rho_{\pm}$ for different fractional orders (the lower plot of Figure \ref{fig:CaDBC:Smoothuf:errcomp:c0}). We observe that using the estimate $\rho_{\pm} = N^{3}$ is enough to obtain high accuracy. Furthermore, from the left plot of Figure \ref{fig:CaFDBC:eig}, which shows the value of $Re(eig)_{min}$ with $p = 0.8, \,0.5$ for $\alpha\in (1,2)$, and the right plot of Figure \ref{fig:CaFDBC:eig}, which shows the values of $Re(eig)_{min}$ against $\rho = \rho_{\pm} $with $p = 0.8$, we can see that coercivity is satisfied by choosing $\rho_{\pm} \ge N^{3}$.
Similar observations can be obtained for $c = 1$, which is not shown here.


\begin{figure}[!t]
\begin{minipage}{0.49\linewidth}
\begin{center}
\includegraphics[scale=0.4,angle=0]{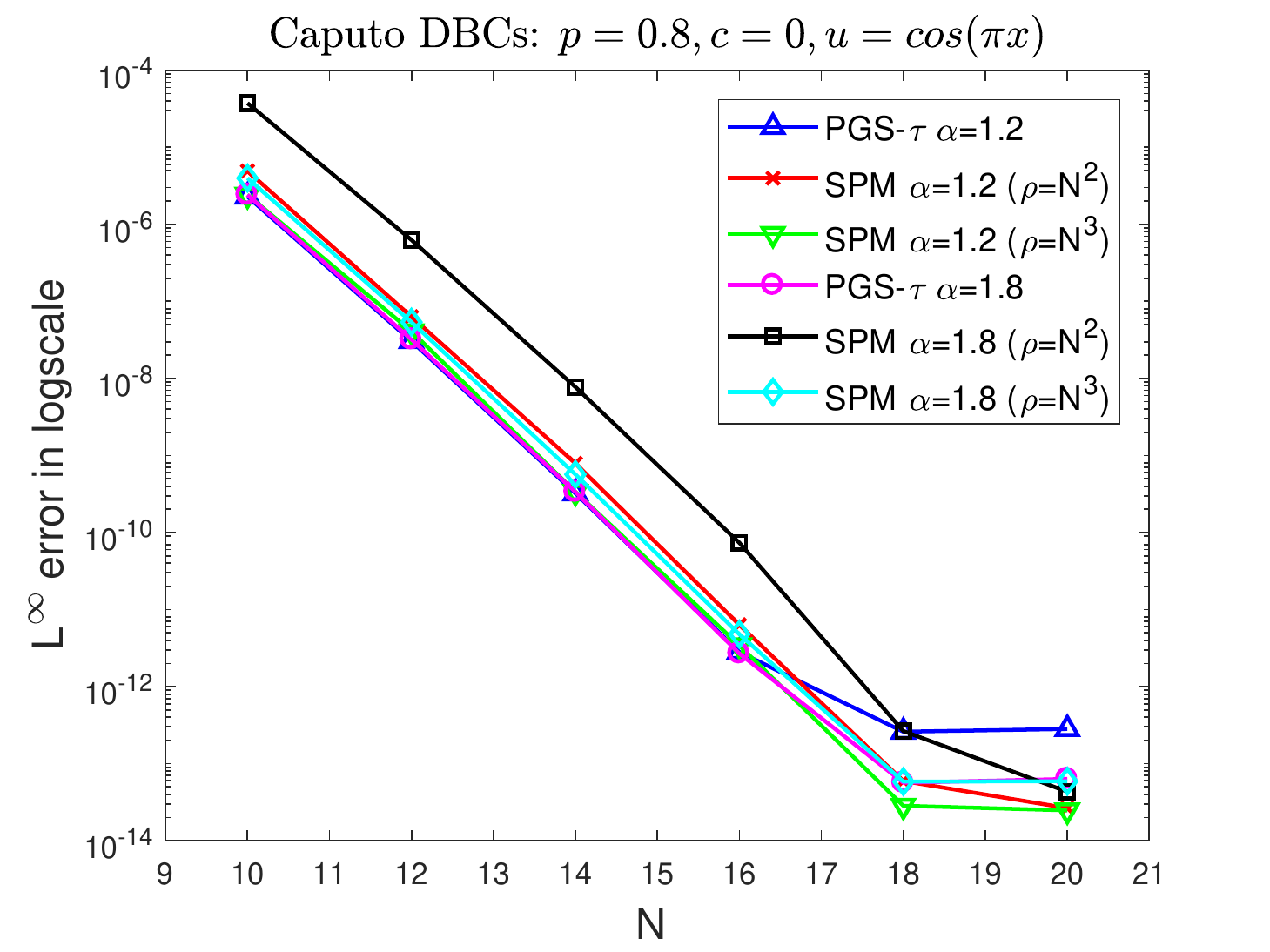}\end{center}
\end{minipage}
\begin{minipage}{0.49\linewidth}
\begin{center}
\includegraphics[scale=0.4,angle=0]{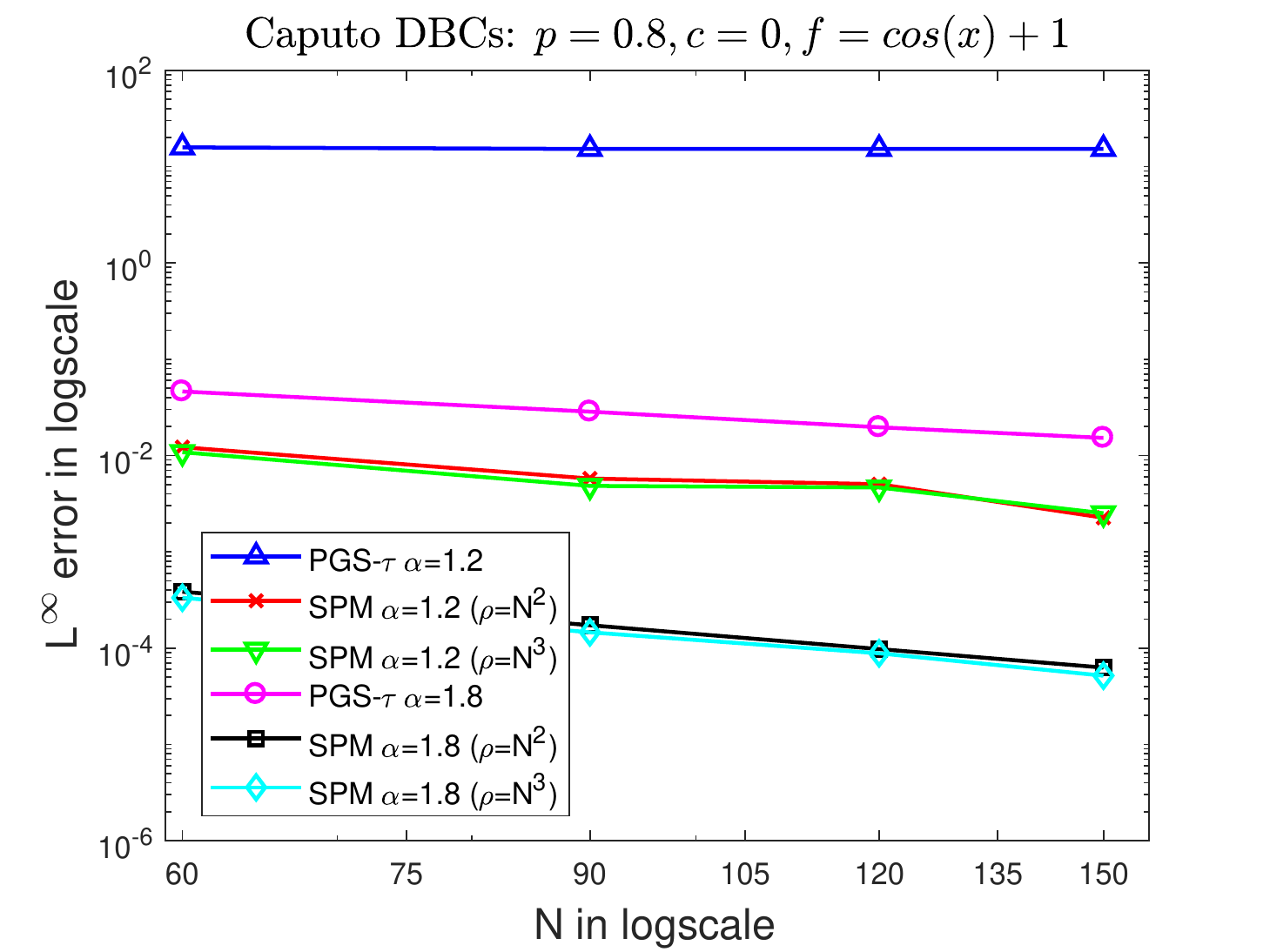}\end{center}
\end{minipage}
\begin{minipage}{0.95\linewidth}
\begin{center}
\includegraphics[scale=0.4,angle=0]{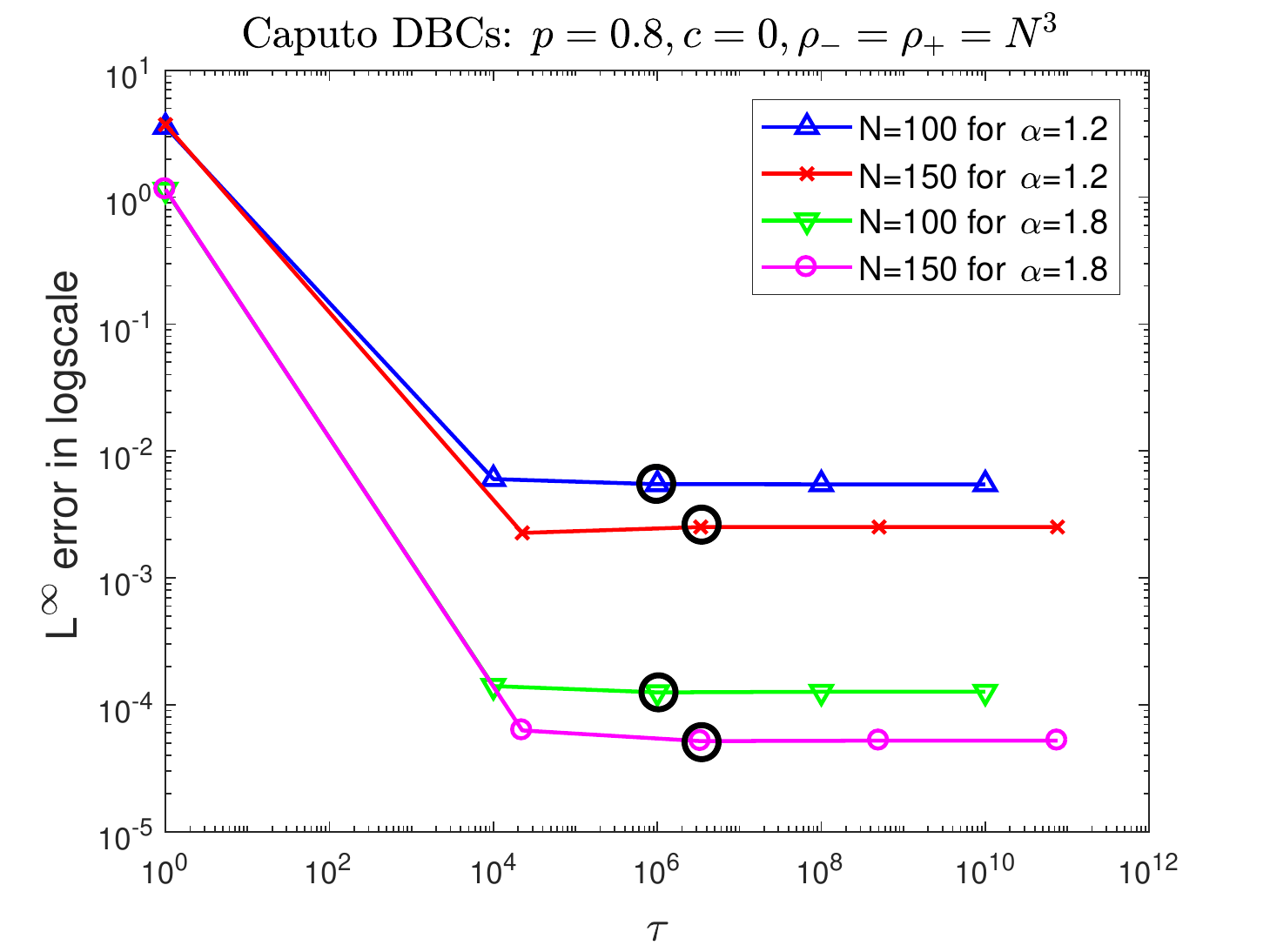}\end{center}
\end{minipage}
\caption{\scriptsize
Example \ref{ex:CaDBC}: Convergence of $L^{\infty}$-error for SPM and PGS-$\tau$ with different values of fractional order $\alpha = 1.2,1.8$. Upper left: versus $N$ for smooth solution $u(x) = \cos(\pi x)$, upper right: versus $N$ for smooth RHF $f(x) = 1+\cos(\pi x)$, lower: versus the penalty parameter $\rho = \rho_{\pm}$. The black circles (down) correspond to $\rho = \rho_{\pm}=N^{3}$.
 }\label{fig:CaDBC:Smoothuf:errcomp:c0}
\end{figure}


\begin{figure}[!t]
\begin{minipage}{0.49\linewidth}
\begin{center}
\includegraphics[scale=0.4,angle=0]{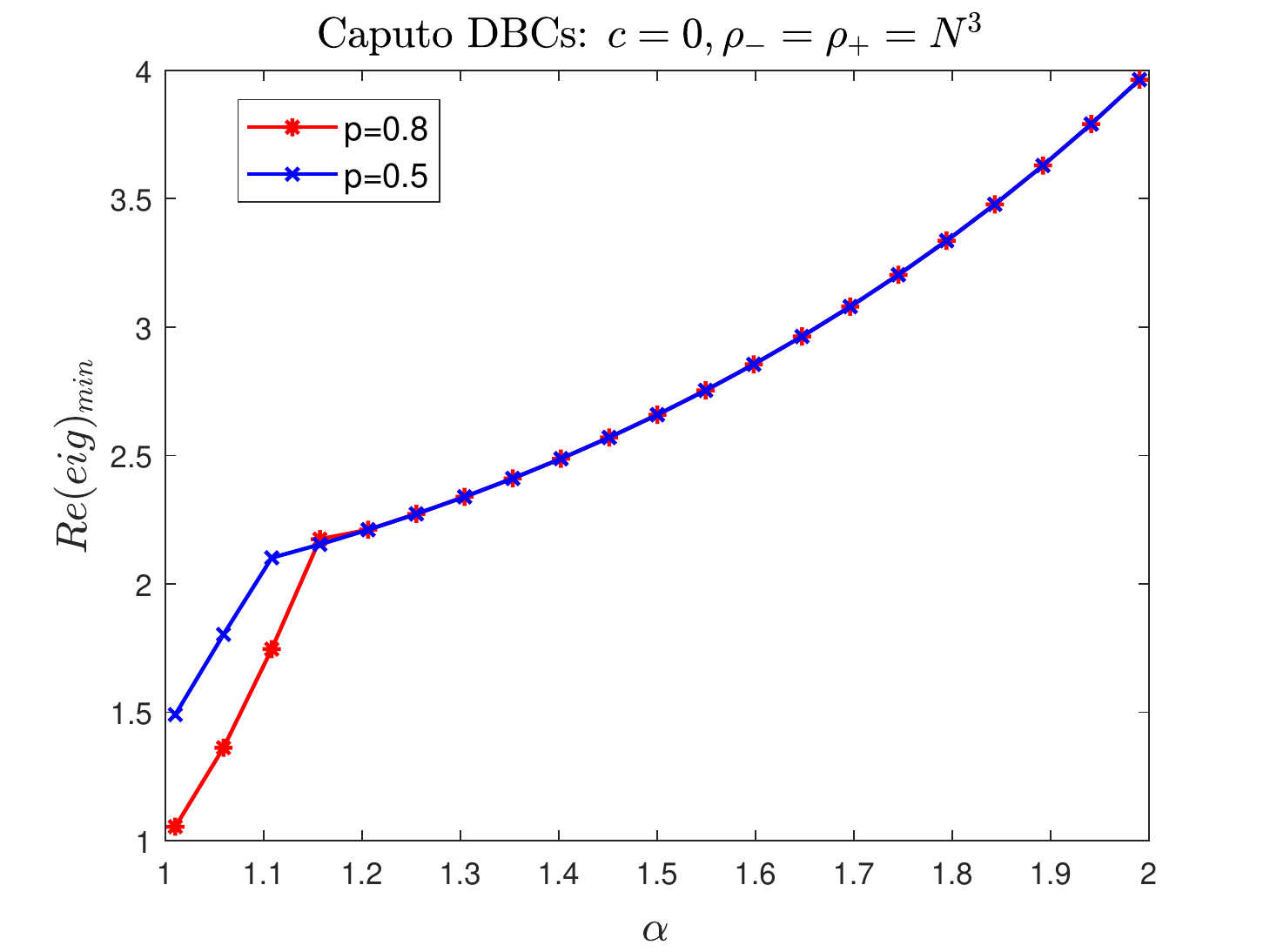}\end{center}
\end{minipage}
\begin{minipage}{0.49\linewidth}
\begin{center}
\includegraphics[scale=0.4,angle=0]{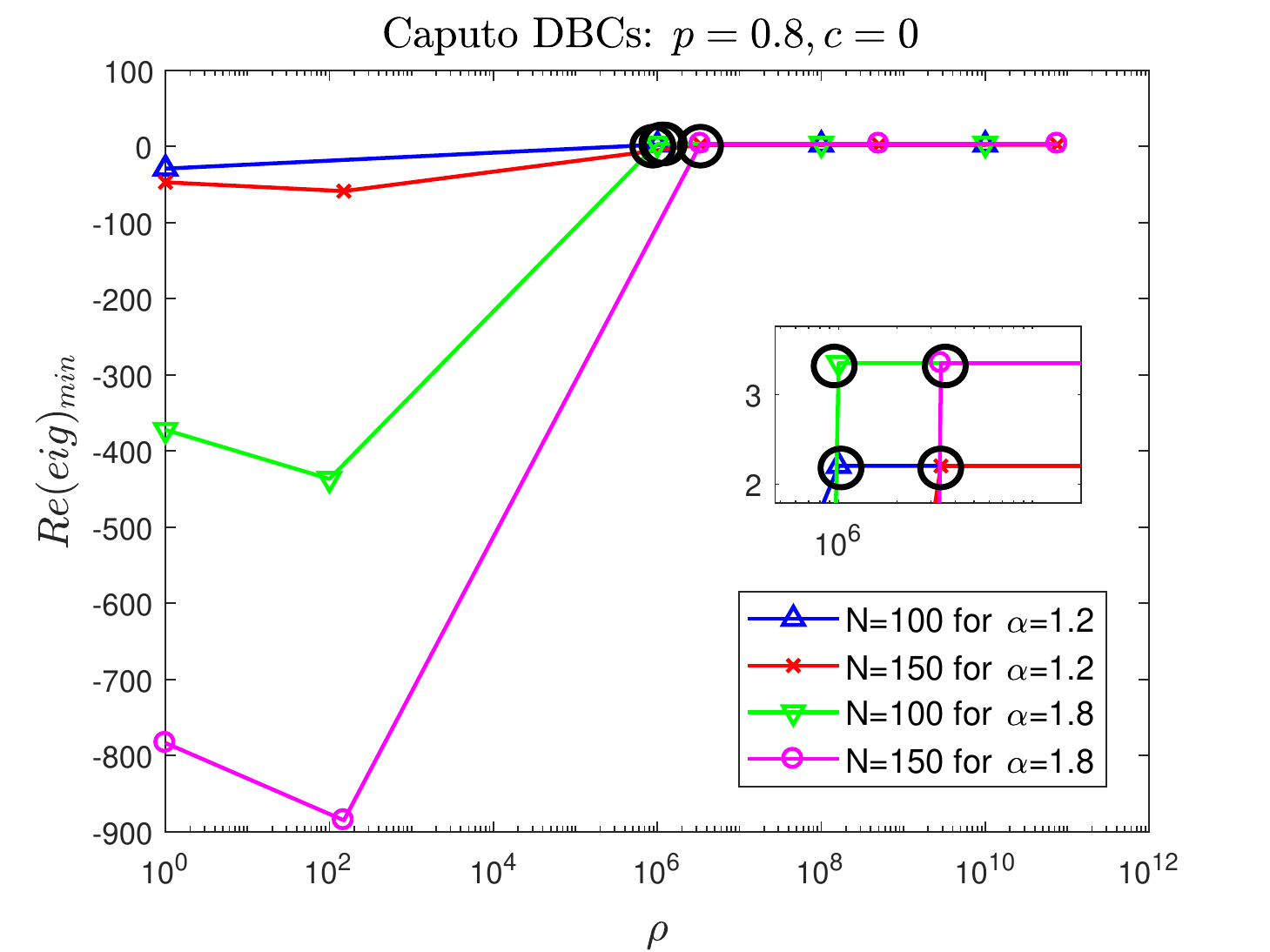}\end{center}
\end{minipage}
\caption{\scriptsize
Example \ref{ex:CaDBC}: The minimum values of the real part  of eigenvalues versus fractional order (left) and penalty parameter $\rho = \rho_{\pm}$ (right). The black circles (right) correspond to $\rho = \rho_{\pm}=N^{3}$.}\label{fig:CaFDBC:eig}
\end{figure}

%
%
%

\begin{exam}\label{ex:CaFNBC}
We now consider the Caputo conservative FDEs with Caputo FNBCs, i.e., \eqref{Intro1.1}-\eqref{IntroFNBC1.5} with the following two cases:
\begin{itemize}
  \item Case I: Smooth solution $u(x) = x^3+1$;
  \item Case II: Smooth RHF $f(x) = 1+\cos(\pi x)$.
\end{itemize}
For Case I, the boundary conditions can be computed directly by the exact solution while for Case II, the boundary conditions are $\dxc^{\a-1}u(-1)=1,\;  \dxc^{\a-1}u(1)=2$.  
\end{exam}

Let $c=1$, $p=0.8$, in view of Theorem \ref{thm:Caputo:FNBC} and Remark \ref{rem5.1}, we  set $\rho_{\pm} = N^{2}$, $Q_N^{\pm}= \pm \frac{1}{N^{2}}\sum_{k=0}^{N} \frac{1}{\gamma_{k}^{0,0}}L_{k}(x)L_{k}(\pm1)$ in this example.
The convergence results of the $L^{\infty}$-error for SPM and PGS-$\tau$ with different values of fractional order $\alpha=1.2,\, 1.8$ for Case I (upper left) and Case II (upper right) are shown in  Figure \ref{fig:CaFNBC:errcomp}. We observe that, same as in the previous example, we obtain spectral accuracy for the Case I and algebraic convergence for the Case II. Also, we obtain \emph{higher} accuracy with SPM than with PGS-$\tau$ in both cases.
We also present the $L^{\infty}$-error for different fractional orders by tuning the penalty parameters $\rho = \rho_{\pm}$ in the lower plot of  Figure \ref{fig:CaFNBC:errcomp}.  From which we observe that the best accuracy is obtained when we choose $\rho_{\pm} = \rho = N^{2}$  satisfying the sufficient condition \eqref{cond:coe:FNBC:C} for the coercivity.
Figure \ref{fig:CaFNBC:Smoothf:eig} shows the value of $Re(eig)_{min}$ against $\alpha$ with $p = 0.8,0.5$ (left) and against $\rho = \rho_{\pm}$ with $p = 0.8$ (right). The results verify the  sufficient condition \eqref{cond:coe:FNBC:C} for the coercivity of SPM \eqref{CDBCPG} with Caputo FNBCs;  coercivity can be maintained if $\rho_{\pm} \ge N^{2}$.

%

\begin{figure}[!tp]
\begin{minipage}{0.49\linewidth}
\begin{center}
\includegraphics[scale=0.4,angle=0]{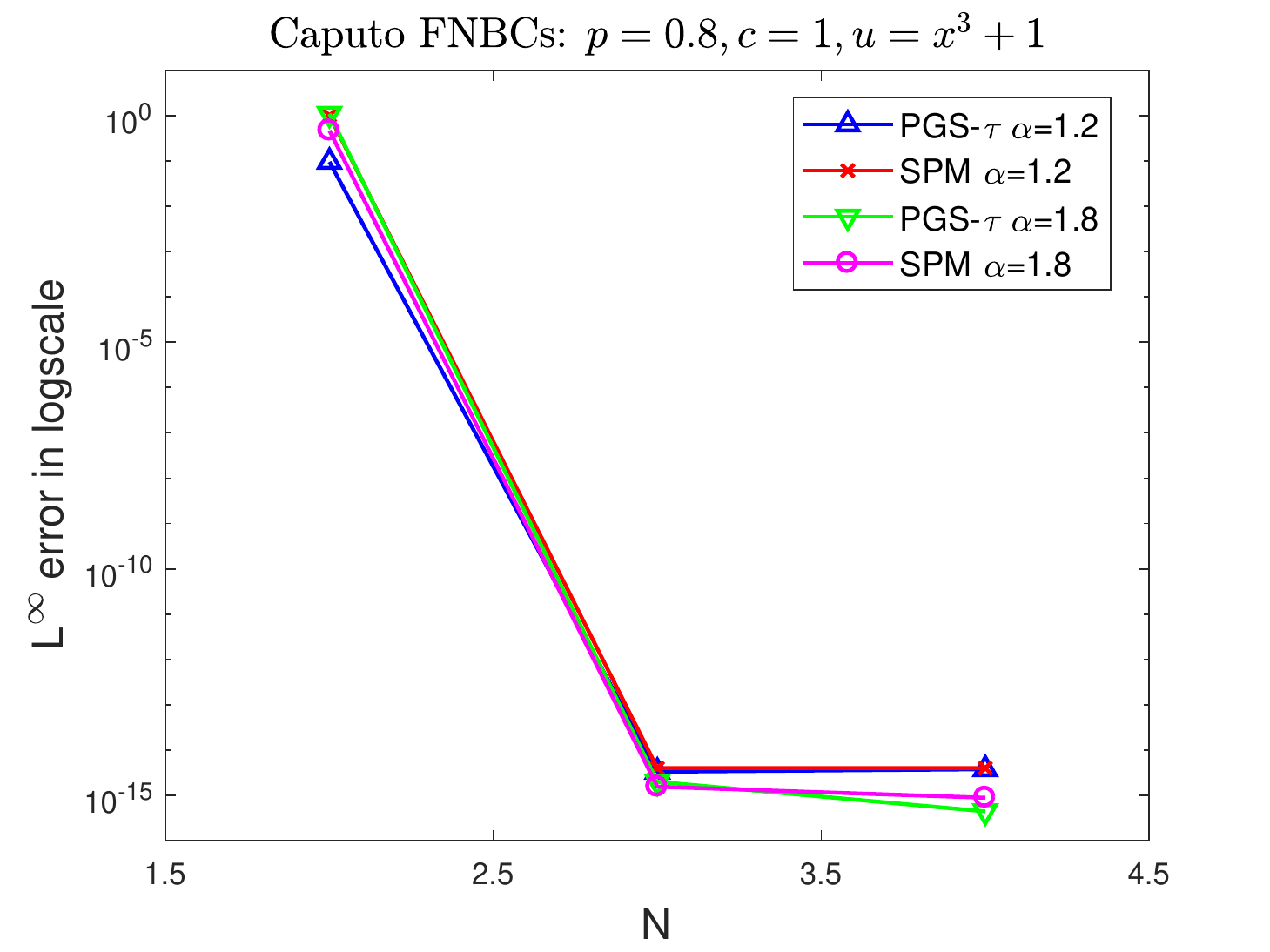}\end{center}
\end{minipage}
\begin{minipage}{0.49\linewidth}
\begin{center}
\includegraphics[scale=0.4,angle=0]{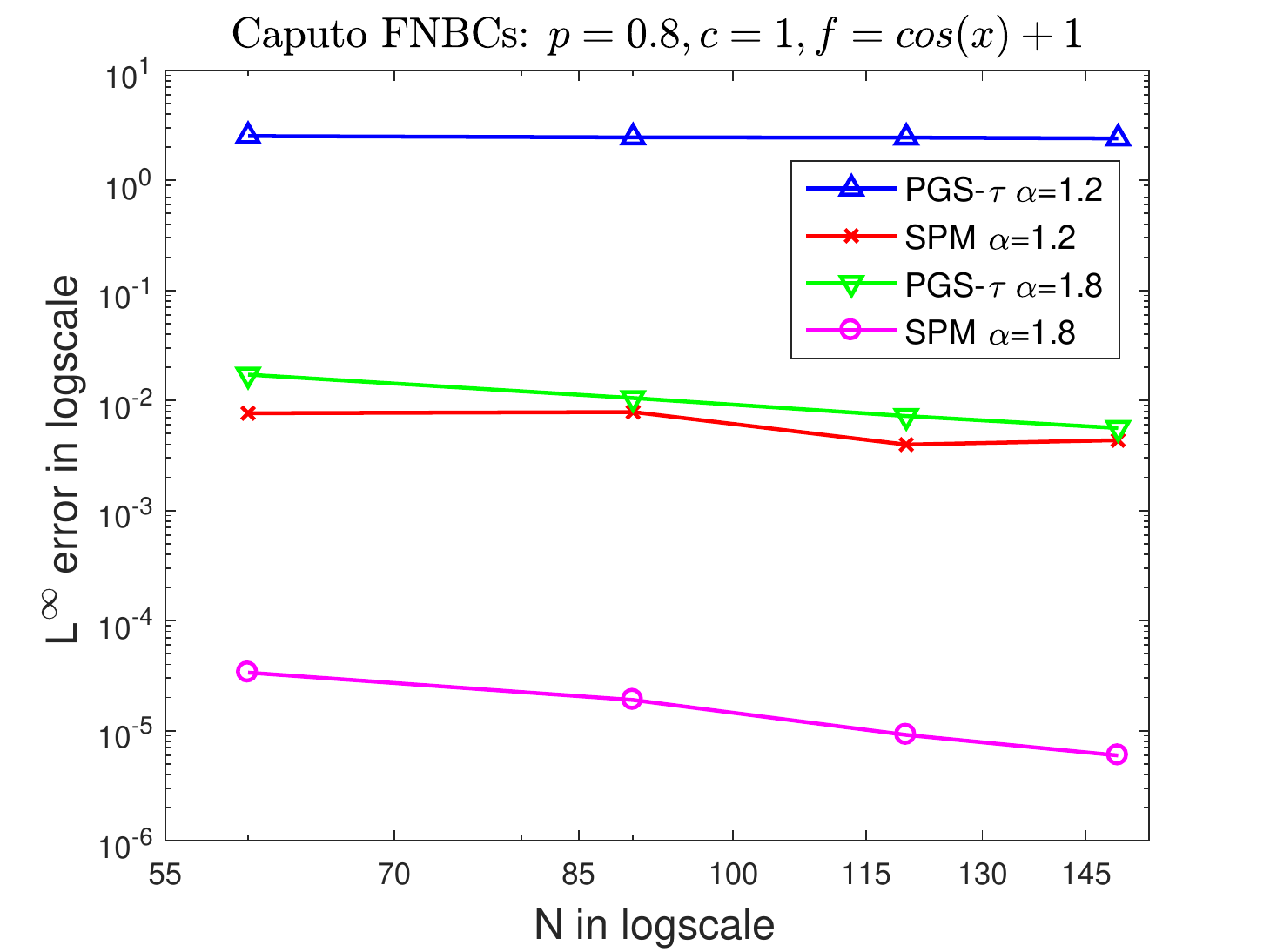}\end{center}
\end{minipage}
\begin{minipage}{0.95\linewidth}
\begin{center}
\includegraphics[scale=0.4,angle=0]{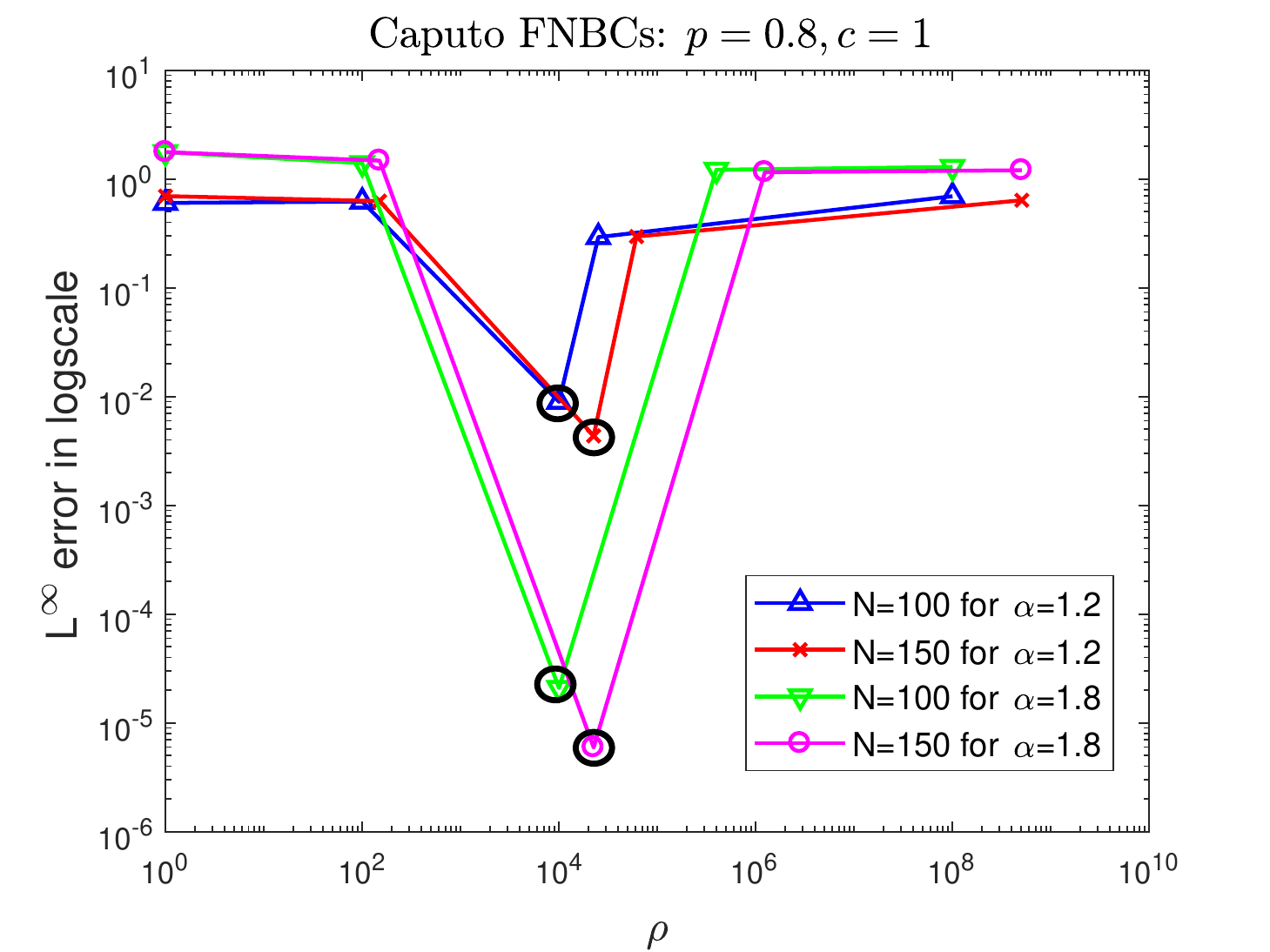}\end{center}
\end{minipage}
\caption{\scriptsize
Example \ref{ex:CaFNBC}: Convergence of $L^{\infty}$-error for SPM and PGS-$\tau$ with different values of fractional order $\alpha = 1.2,1.8$. Upper left: versus $N$ for smooth solution $u(x) = x^3+1$, upper right: versus $N$ for smooth RHF $f(x) = 1+\cos(\pi x)$, lower: versus the penalty parameter $\rho = \rho_{\pm}$. The black circles (down) correspond to the penalty parameters satisfying the coercivity sufficient condition \eqref{cond:coe:FNBC:C}.}\label{fig:CaFNBC:errcomp}
\end{figure}

%

%
%
%

\begin{figure}[!t]
\begin{minipage}{0.49\linewidth}
\begin{center}
\includegraphics[scale=0.4,angle=0]{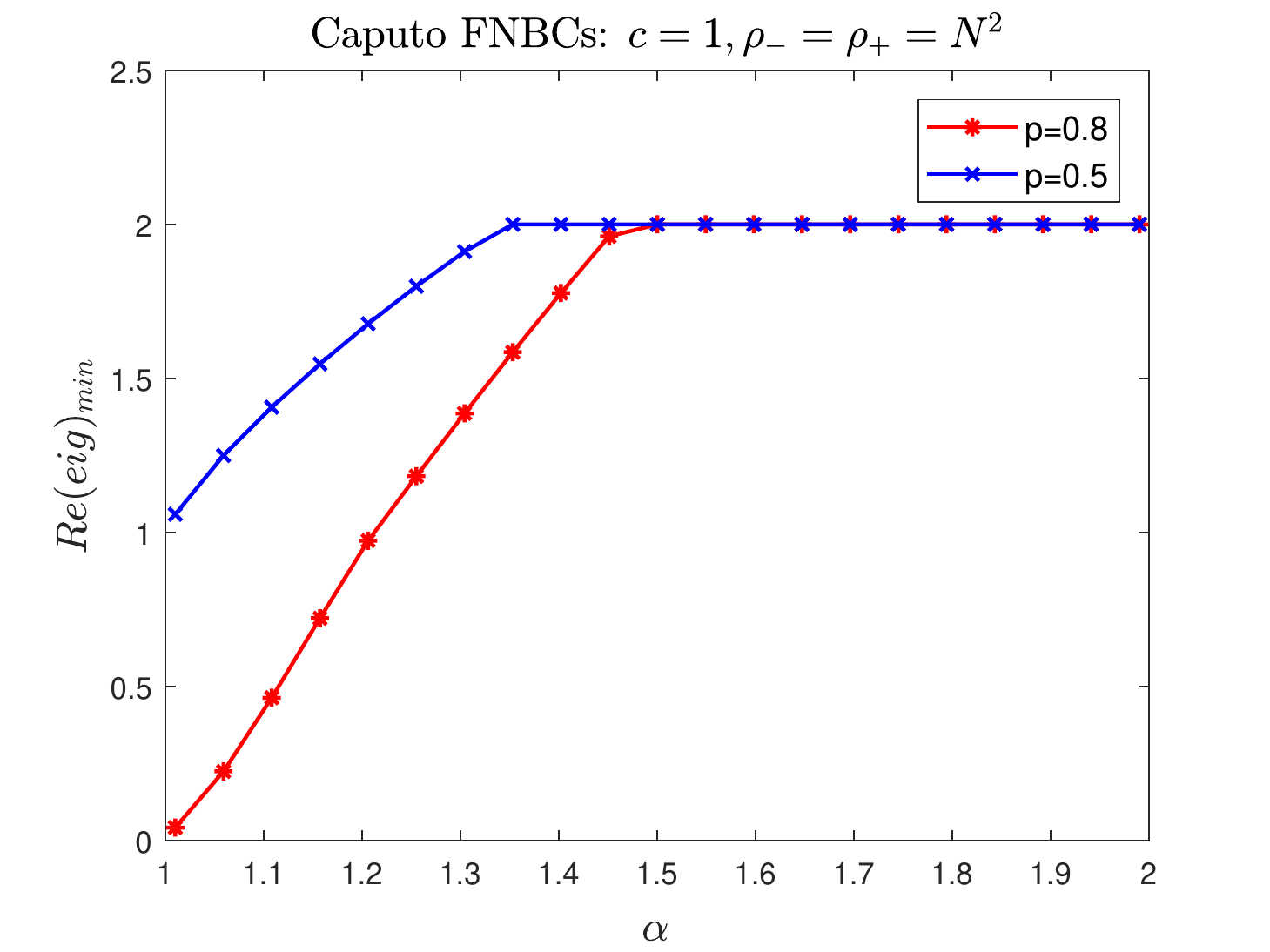}\end{center}
\end{minipage}
\begin{minipage}{0.49\linewidth}
\begin{center}
\includegraphics[scale=0.4,angle=0]{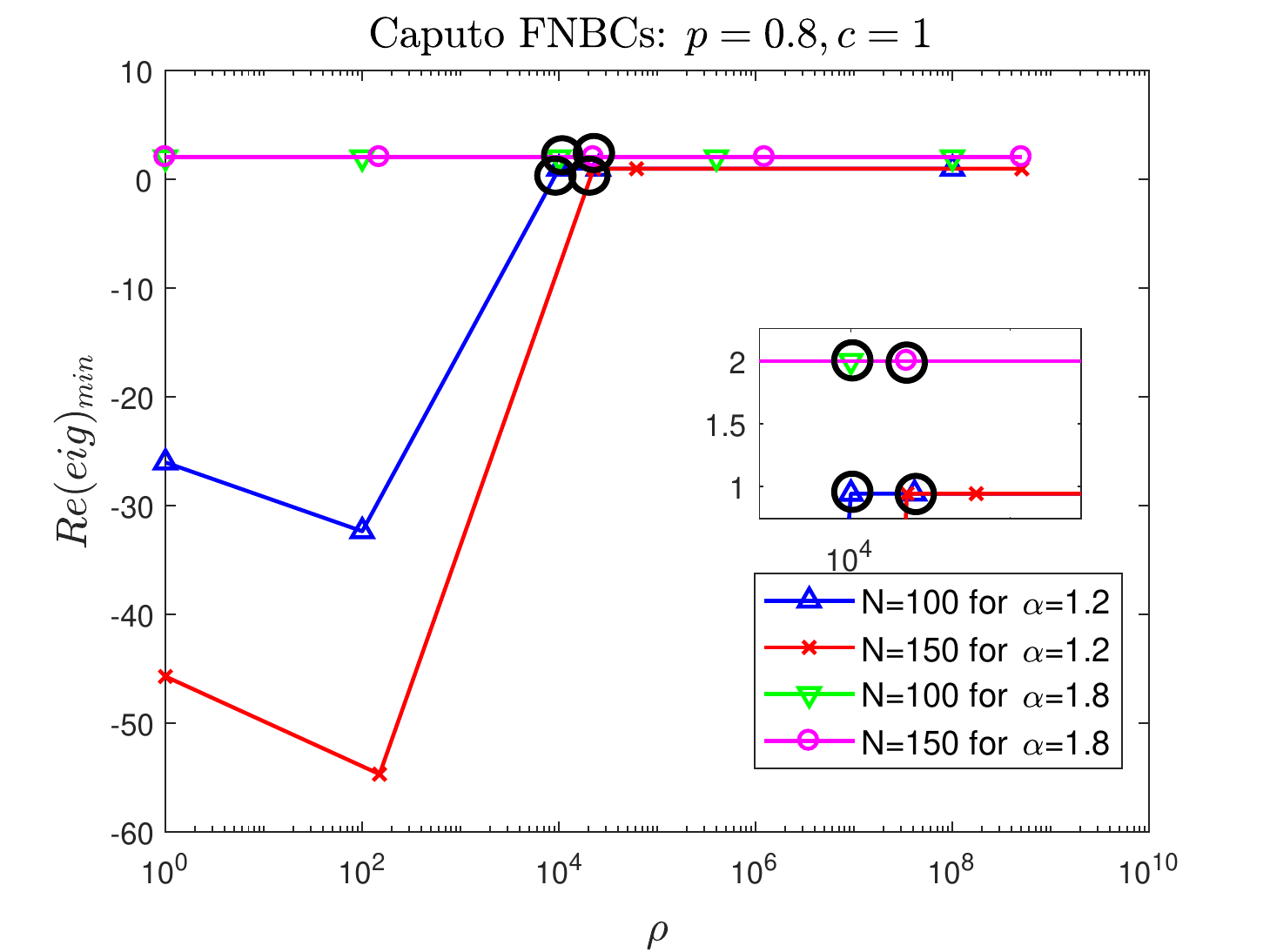}\end{center}
\end{minipage}
\caption{\scriptsize
Example \ref{ex:CaFNBC}: The minimum values of the real part  of eigenvalues versus fractional order (left) and penalty parameter $\rho = \rho_{\pm}$ (right). The black circles (right) correspond to the penalty parameters satisfying the coercivity sufficient condition \eqref{cond:coe:FNBC:C}.}\label{fig:CaFNBC:Smoothf:eig}
\end{figure}



\section{Application to the time dependent problem}\label{sec:application}
We finally solve a time dependent two-sided fractional diffusion equation considered in \cite{KellyPreprint} with reflecting (no-flux) BCs,  i.e., homogeneous R-L/Caputo FNBCs,  by using SPM.
\begin{exam}\label{ex:Twosidediffusion}
Consider equation \eqref{Twosidediffusion1}
with homogeneous FNBCs, i.e., $\mathbb{D}_{x}^{\alpha-1} u(x,t)$ $= 0$, and a tent function
\begin{align*}
  u_{0}(x)=\left\{
             \begin{array}{ll}
           5-25|x|, \quad &|x|<0.2,\\
           0, &\text{otherwise}
             \end{array}
           \right.
\end{align*}
as the initial condition with mass $M_{0}=1$, where $M_{0}=\int_{-1}^{1}u(x,t)dx$.
\end{exam}

In the numerical simulations, we use SPM for space discretization and the first order implicit Euler scheme for time discretization. Let $\delta t = T/K$ be the time step, then for $n = 0,1,\ldots, K-1$, the fully discrete scheme for \eqref{Twosidediffusion1}  is to find $u^{n+1}_{N} \in \mathbb{X}_N$, such that
\begin{equation*}
\begin{aligned}
 &\left(\frac{u_{N}^{n+1}-u_{N}^{n}}{\delta t},\varphi(v_N) \right)- \left(\frac{d}{dx}\mathbb{D}_{x}^{\alpha-1}u^{n+1}_{N}, \varphi(v_N) \right) \\
 = & -\rho_{-}\mathbb{D}_{x}^{\alpha-1}u^{n+1}_{N}(-1)(Q_{N}^{-},\varphi(v_N))
-\rho_{+}\mathbb{D}_{x}^{\alpha-1}u^{n+1}_{N}(1)(Q_{N}^{+},\varphi(v_N))\quad v_{N}\in \mathbb{X}_N, \\
\end{aligned}
\end{equation*}
where for the  R-L case $\mathbb{X}_N = \mathcal{F}_{N}^{-\mu,-\nu},\; \varphi(v_N) = \ix^{2-\a}v_{N}$ while for the Caputo case $\mathbb{X}_N = \mathbb{P}_N,\; \varphi(v_N) =v_{N}$.

For the R-L problem, the penalty functions are chosen to be the same as for Example \ref{ex:RLFNBC} and the penalty parameters are chosen to be $\rho_{-}=N^{2\mu+2}, \rho_{+}=N^{2\nu+2}$. For the Caputo problem, the penalty functions are chosen to be the same as for Example \ref{ex:CaFNBC} and the penalty parameters are chosen to be $\rho_{\pm}=N^{2}$.
For $\a = 1.5$, the numerical solutions obtained by using  $N = 100$ and time step $\delta t = 0.0025$ at different times $t = 0,\, 0.05,\,0.1,\, 2$ are shown in Figure \ref{fig:difusion:num}. The left panel ($p=0.75$) is for the R-L problem while the right panel ($p=0.25$) is for the Caputo problem. The numerical results are consistent with the observation in \cite{KellyPreprint}, where the fractional diffusion equation is solved by a finite difference method with space size $N = 1000$.  Here we show that in both cases the numerical solutions tend to the steady states. Moreover, the steady state for diffusion with R-L flux exhibits boundary singularities. In contrast,  the steady state for diffusion with Caputo flux is a constant $u_\infty = 1/2$. Moreover, we found that the mass is \emph{conserved} at all times.

\begin{figure}[!t]
\begin{minipage}{0.49\linewidth}
\begin{center}
\includegraphics[scale=0.4,angle=0]{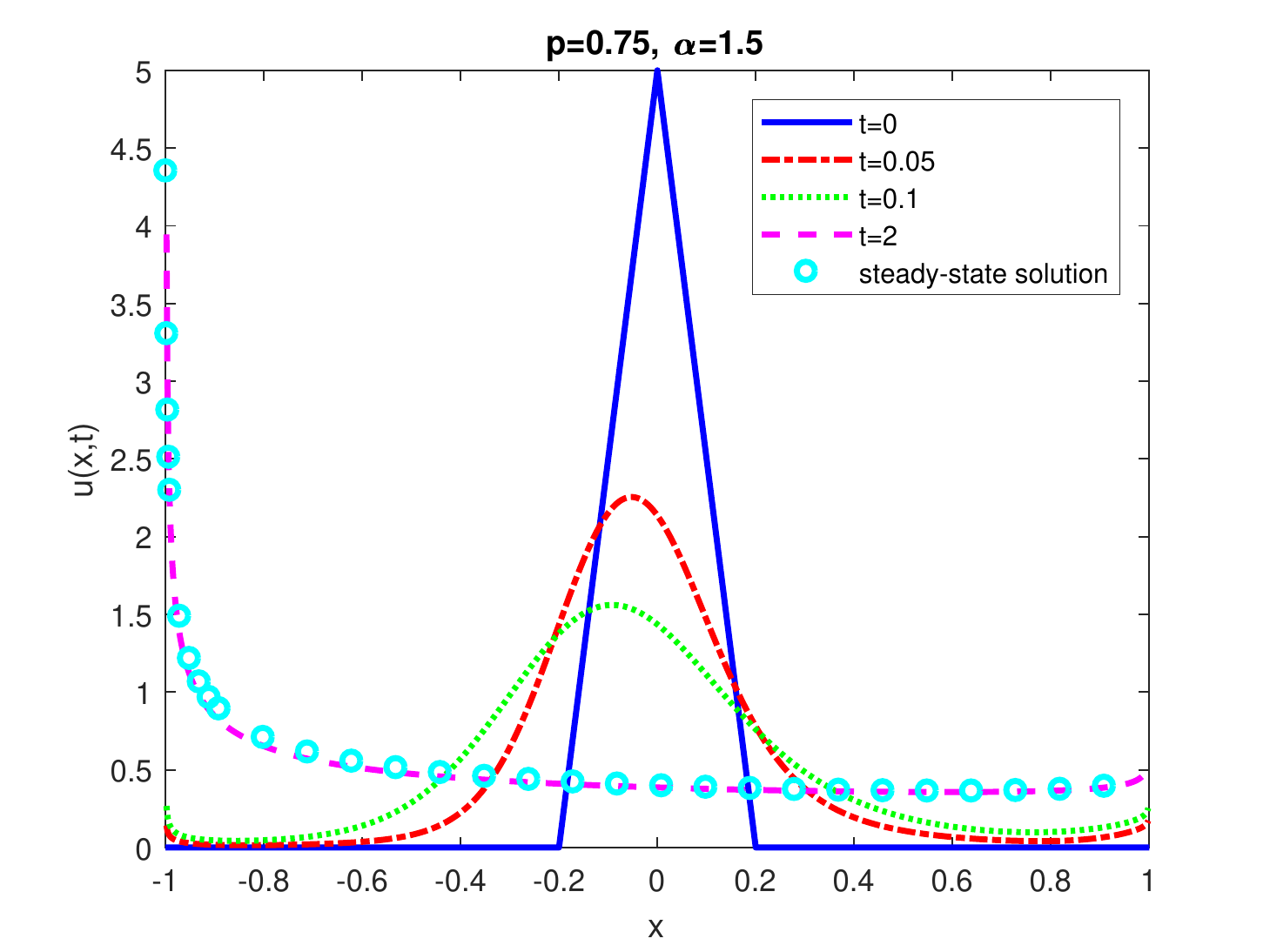}\end{center}
\end{minipage}
\begin{minipage}{0.49\linewidth}
\begin{center}
\includegraphics[scale=0.4,angle=0]{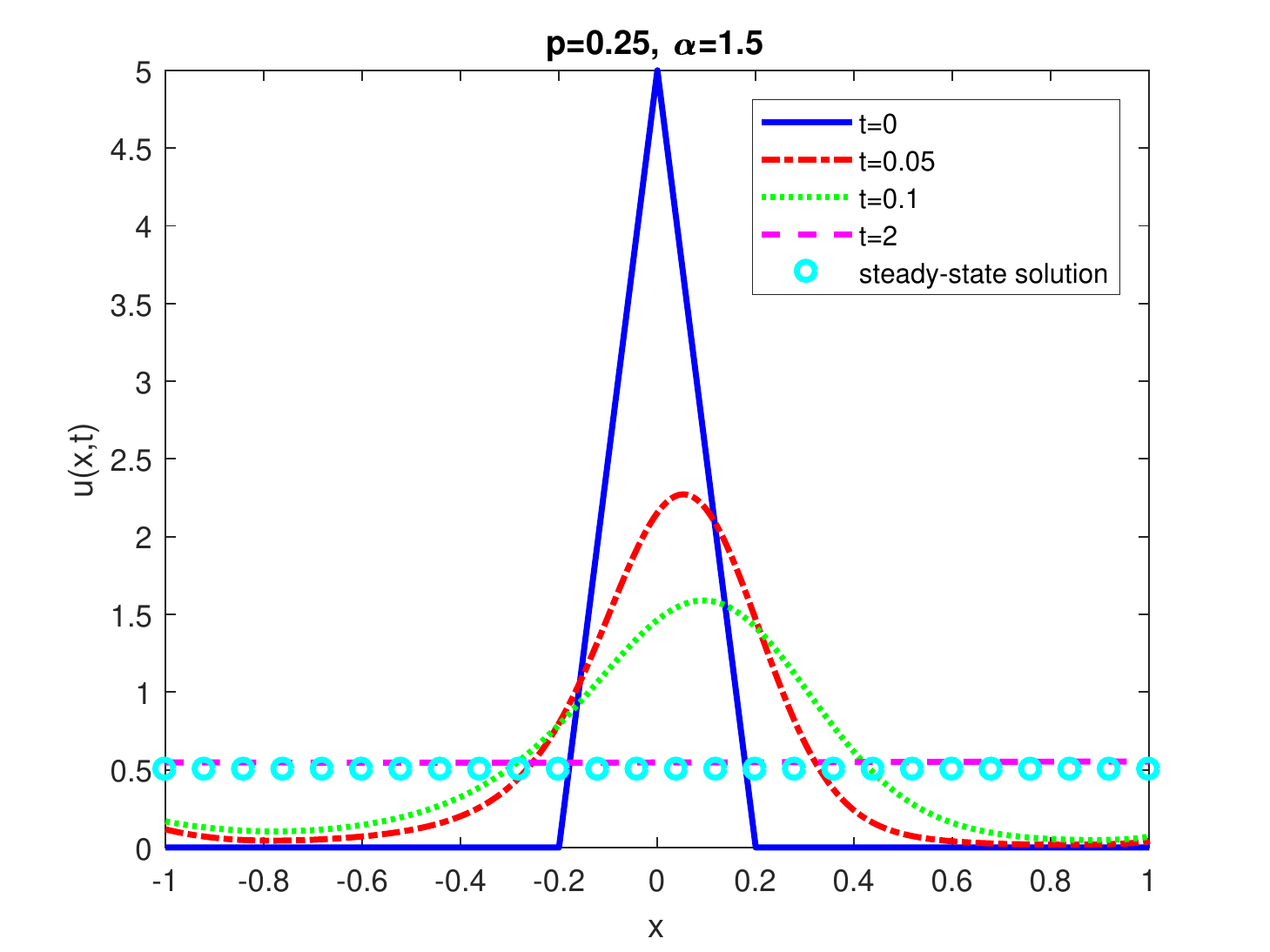}\end{center}
\end{minipage}
\caption{\scriptsize
Example \ref{ex:Twosidediffusion}: Numerical solutions  of fractional diffusion equation with $\a=1.5$ at different times $t=0, \,0.05,\, 0.1,\,2$, and the steady-state solution. Left: R-L problem ($p = 0.75$); right: Caputo problem ($p = 0.25$). }\label{fig:difusion:num}
\end{figure}

\section{Conclusion}\label{sec:conclusion}
We proposed in this paper a Galerkin spectral penalty method (SPM) for the two-sided FDE with \emph{general} BCs using poly-fractonomial or polynomial approximation. Specifically, we used orthogonal poly-fractonomials, whose spectral relationship with fractional operators has been documented in \cite{MohsenKar13, MaoKar18}, as basis functions for the conservative R-L FDEs while used orthogonal polynomials as basis functions for the conservative Caputo FDEs.
We first established the well-posedness of the weak problem of the conservative R-L problem with the fractional Dirichlet/Neumann boundary conditions.
Subsequently, we formulated SPM for the conservative R-L and Caputo FDEs.
We also analyzed sufficient conditions for the coercivity of different types of fractional problems, and moreover provided estimates of the penalty parameters and the associated functions for all cases except for Caputo FDEs with local Dirichlet boundary conditions.
We showed by several numerical examples that SPM can deliver superior  accuracy  compared with the Petrov-Galerkin spectral tau method, and verified the theoretical estimates for the sufficient conditions for {coercivity} as well as  the estimates for  the penalty parameters. For the aforementioned case not covered by the theory, we conducted numerical experiments and proposed penalty parameters that scale as $N^2$ or $N^3$, where $N$ is the polynomial order. In general, as long as we choose the value of the penalty parameter greater than the threshold suggested by the theory, the accuracy of the approximation does not depend on the precise value of the penalty parameters for the Dirichlet type boundary conditions. In contrast, the approximation error for the Neumann type boundary conditions depends strongly on the specific value of the penalty parameter and increases sharply away from the theoretical value, especially  for Caputo FDEs.
Finally, we solved the time dependent  fractional diffusion equation by using  SPM and verified the conservation property, hence confirming the accurate imposition of fractional Neumann boundary conditions for this case.

Overall, we found that in the absence of reaction term in the FDEs, we can obtain exponential decay of the numerical error for smooth right hand side in the case of R-L FDEs, irrespectively of the boundary conditions although the solution in this case is singular. In contrast, we found that we can obtain  exponential decay of the numerical error for smooth solutions of Caputo FDEs, irrespectively of the boundary conditions. We note that in the latter case, exponential convergence can be obtained even in the presence of reaction term.
Here we considered one-dimensional FDEs, but the penalty implementation can be also extended to multi-dimensions as well as other discretizations (collocation, finite elements, etc.). However, large values of the penalty parameter may adversely  affect the condition number of the linear system and hence the computational complexity for iterative solvers in large scale problems.
Also, the presence of a penalty term may destroy the sparsity of the stiffness matrix obtained in \cite{MaoKar18} when using the poly-fractonomial approximation.

\appendix

\section{Technical results}\label{sec:Tech:RePr}

We collect a number of elementary technical results that were used to prove the well-posedness of the weak problem \eqref{wk:fDBC} and \eqref{wk:fNBC}.

Define the fractional integral space and norm: for $\sigma \ge 0$
\begin{equation*}
    J^{-\sigma}(\Lambda) := \left\{ v: |\omega|^{-\sigma} \mathcal{F}^c(v) \in L^2(\mathbb{R})\right\},
\;
    \|v\|_{J^{-\sigma}(\Lambda)} : = \||\omega|^{-\sigma} \mathcal{F}^c(v)\|_{L^2(\mathbb{R})},
\end{equation*}
where $\mathcal{F}^c$ denotes the incomplete Fourier transform given by (cf. \cite[Equation (18)]{Ma.J2017JSC})
\begin{equation*}
    \mathcal{F}^c(v):= \frac{1}{2\pi}\int_{\Lambda} v(x)e^{{\rm i}wx} dx.
\end{equation*}
Denote $\widetilde{v}(x), x\in\mathbb{R}$ the zero extension of $v(x),x\in [a,b]$, namely, $\widetilde{v}(x) = v(x)$ if $x\in [a,b]$ and 0 otherwise.
For $\sigma\ge 0$, we have the following result.
\begin{Lem}\label{lem:equv}
  For $\sigma\ge 0$, we have
\begin{equation}\label{eqn:equv}
    (\ilx^{\sigma}v, \irx^{\sigma}v) = \cos(\pi \sigma)\|v\|_{J^{-\sigma}(\Lambda)}^2.
\end{equation}
\end{Lem}
\begin{proof}
Let $\widetilde{v}(x)$ be the zero extension of $v(x)$, then
we have
\begin{equation*}
\begin{aligned}
  (\ilx^{\sigma}v, \irx^{\sigma}v) = (\ilxinf^{\sigma}\widetilde{v}, \irxinf^{\sigma}\widetilde{v})
    &= \cos(\pi \sigma) \||\omega|^{-\sigma} \mathcal{F}(\widetilde{v})\|_{L^2(\mathbb{R})}^2\\
    & = \cos(\pi \sigma)\||\omega|^{-\sigma} \mathcal{F}^c(v)\|_{L^2(\mathbb{R})}^2,
\end{aligned}
\end{equation*}
where $\mathcal{F}(\cdot)$ is the Fourier transform.
The second equality of the above equation can be found in \cite[Lemma 2.3]{Ma.J2017JSC}.
Thus, the equality \eqref{eqn:equv} holds true.
\end{proof}

The following lemma shows that the space $ J^{-\sigma}(\Lambda)$ is embedded into the spaces $J_{l}^{-\sigma}(\Lambda)$ and $J_{r}^{-\sigma}(\Lambda)$: \begin{Lem}\label{lem:intequvi}
For $\sigma \ge 0$, the space $ J^{-\sigma}(\Lambda)$ is embedded into the spaces $J_{l}^{-\sigma}(\Lambda)$ and $J_{r}^{-\sigma}(\Lambda)$ satisfying
\begin{equation}\label{eqn:intequvi}
    \|v\|_{J_{l}^{-\sigma}(\Lambda)} \le \|v\|_{J^{-\sigma}(\Lambda)} \text{ and } \|v\|_{J_{r}^{-\sigma}(\Lambda)} \le \|v\|_{J^{-\sigma}(\Lambda)}.
\end{equation}
\end{Lem}
\begin{proof}
%
For the first estimate of \eqref{eqn:intequvi}, we have
\begin{equation*}
  \|v\|_{J_{l}^{-\sigma}(\Lambda)} = \|\ilx^{\sigma}v\|_{L^2(\Lambda)}
  \le  \|\ilxinf^{\sigma}\widetilde{v}\|_{L^2(\mathbb{R})}
    = \||\omega|^{-\sigma} \mathcal{F}^c(v) \|_{L^2(\mathbb{R})}, 
\end{equation*}
where the last equality  follows from \cite[equation (25)]{Ma.J2017JSC}. Then we obtain the first estimate of \eqref{eqn:intequvi}.
Similarly, we can obtain the second estimate of \eqref{eqn:intequvi}.
\end{proof}

We show the following result of the boundedness of the fractional integral operators $\ilx^{\sigma}$ and $\irx^{\sigma}$ for $\sigma \ge 0$~(see \cite[Theorem 2.6]{samko1993fractional}):
\begin{Lem}\label{lem:ixbdd}
The fractional integral operators $\ilx^{\sigma}$ and $\irx^{\sigma}$ with $\sigma \ge 0$ are bounded in $L^q(\Lambda)\; (1\le q\le \infty)$:
\begin{equation}\label{eqn:ixbdd}
    \|\ilx^{\sigma} v\|_{L^q(\Lambda)} \le K(\sigma)\|v\|_{L^q(\Lambda)}, \text{ and } \|\irx^{\sigma} v\|_{L^q(\Lambda)} \le K(\sigma)\|v\|_{L^q(\Lambda)},
\end{equation}
where $K(\sigma) = \frac{2^{\sigma} }{\sigma |\Gamma(\sigma)|}$.
\end{Lem}

For $0\le s\le t$, we have the following result:
\begin{Lem}
For $0\le s\le t$, it holds
\begin{equation}\label{eqn:intcoerlr}
    \|v\|_{J_{l}^{-t}(\Lambda)} \le C\|v\|_{J_{l}^{-s}(\Lambda)}, \quad  \|v\|_{J_{r}^{-t}(\Lambda)} \le C\|v\|_{J_{r}^{-s}(\Lambda)}.
\end{equation}
where $C$ is a constant.
Moreover, if $s<1/2$, we have
\begin{equation}\label{eqn:intcoer}
    (\ilx^s v, \irx^s v) \ge \tilde{C}\|v\|_{J_{p}^{-t}(\Lambda)}^2 ,
\end{equation}
where $\tilde{C}$ is a  constant.
\end{Lem}
\begin{proof}
We begin by showing the first estimate of \eqref{eqn:intcoerlr}.
By virtue of \eqref{eqn:ixbdd} and \eqref{eqn:semig}, we arrive at
\begin{equation*}
    \|v\|_{J_{l}^{-t}(\Lambda)}=\|\ilx^{t}v\|_{L^2(\Lambda)} = \|\ilx^{t-s}\, \ilx^{s}v\|_{L^2(\Lambda)} \le C  \|\ilx^{s}v\|_{L^2(\Lambda)} = C\|v\|_{J_{l}^{-s}(\Lambda)}.
\end{equation*}
Similarly, we can obtain the second estimate of \eqref{eqn:intcoerlr}.

We now turn to the estimate \eqref{eqn:intcoer}.
By \eqref{defn:fintspl:norm}, \eqref{defn:fintsps}, the definition of the two-sided fractional integral and the estimates \eqref{eqn:intequvi} and \eqref{eqn:intcoerlr}, we arrive at
\begin{equation*}
\begin{aligned}
  \|v\|_{J_{p}^{-t}(\Lambda)}^2  \le 2\big(p^2\|v\|_{J_{l}^{-t}(\Lambda)}^2 + (1-p)^2 \|v\|_{J_{r}^{-t}(\Lambda)}^2 \big)
    &\le C_1\|v\|_{J_{l}^{-s}(\Lambda)}^2 + C_2\|v\|_{J_{r}^{-s}(\Lambda)}^2\\
   &\le C \|v\|_{J^{-s}(\Lambda)}^2.
\end{aligned}
\end{equation*}
Then, the estimate \eqref{eqn:intcoer} follows from the equality \eqref{eqn:equv}.
\end{proof}

\section{Petrov-Galerkin spectral tau method (PGS-$\tau$)}\label{sec:app:tau}
For the sake of completeness, we present in this appendix PGS-$\tau$ method.

\subsection{PGS-$\tau$ for the conservative R-L FDEs}
We first establish PGS-$\tau$ for the conservative R-L FDEs, i.e., \eqref{Intro1.1}-\eqref{IntroFDBC1.2} or \eqref{Intro1.1}-\eqref{IntroFNBC1.3}.
The PGS-$\tau$ for \eqref{Intro1.1}-\eqref{IntroFDBC1.2} or \eqref{Intro1.1}-\eqref{IntroFNBC1.3} is to find $u_{N} \in \mathcal{F}_{N}^{-\mu,-\nu}$ such that
\begin{equation}\label{RLPG:Tau}
    A_{T}^{R-L}(u_N,v) = (f,\ix^{2-\a}v)_{\omega} \quad \forall  v\in \mathcal{F}_{N-2}^{-\mu,-\nu}; \quad
    \mathcal{B}_{-}u_N(-1) = g_1, \; \mathcal{B}_{+}u_N(1) = g_2,
\end{equation}
where the bilinear form $A_T^{R-L}(\cdot,\cdot)$ is given by
\begin{equation*}
  A_T^{R-L}(u_N,v) : = -(\dx^{\alpha}u_N,\ix^{2-\a}v)_{\omega}+c(u_N,\ix^{2-\a}v)_{\omega}
\end{equation*}
and  $\mathcal{B}_{\pm}u_N(\pm1)=\ix^{2-\alpha}u_N(\pm1)$ for FDBC \eqref{IntroFDBC1.2} while $ \mathcal{B}_{\pm}u_N(\pm1)=\dx^{\alpha-1}u_N(\pm1)$ for FNBC \eqref{IntroFNBC1.3}.

By taking $u_N(x)=\sum_{k=0}^{N}\tilde{u}_{k}J_{k}^{-\mu,-\nu}$, and
letting the test functions  be $J_{i}^{-\mu,-\nu}(x)$, $0\leq i\leq N-2$,  we obtain the following linear system
\begin{equation*}
  (-\widetilde{S} + c\widetilde{M} + \widetilde{B}){U}=F,
\end{equation*}
where ${U}=(\tilde{u}_0,\tilde{u}_1,\cdots,\tilde{u}_{N})^T$, the stiffness and mass matrix $\widetilde{S}$ and $\widetilde{M}$ have the same elements as the matrix $S$ and $M$ in \eqref{RLLMS} except that the last two rows are equal to zero. The first $N-1$ rows of the matrix $\widetilde{B}$ are equal to zero, the last two rows of the matrix $\widetilde{B}$ are given by
$$\widetilde{B}_{N-1,k} = (\mathcal{B}_{-}J_{k}^{-\mu,-\nu}) (-1), \; \widetilde{B}_{N,k} = (\mathcal{B}_{+}J_{k}^{-\mu,-\nu}) (1), $$
and $F = [\widetilde{F}(0:N-1);g_1;g_2]$ where $\widetilde{F}$ is given by \eqref{RLLMS}.
%

\subsection{PGS-$\tau$ for the conservative Caputo FDEs}
The PGS-$\tau$ for the conservative Caputo FDEs, i.e., \eqref{Intro1.1}-\eqref{IntroDBC1.4} or \eqref{Intro1.1}-\eqref{IntroFNBC1.5} is to find $u_N\in \mathbb{P}_N$, such that
\begin{equation}\label{CDBCPG:Tau}
 A_T^{C}(u_N,v) = (f,v)_{\omega} \quad \forall \, v\in \mathbb{P}_{N-2}; \quad
 \mathcal{B}_{-}u_N(-1) = g_1, \; \mathcal{B}_{+}u_N(1) = g_2,
\end{equation}
where the bilinear form $A_T^{C}(\cdot,\cdot)$ is given by
\begin{equation*}
  A_T^{C}(u_N,v) : = -(D\, \dxc^{\alpha-1}u,v)_{\omega}+c(u_N,v)_{\omega}
\end{equation*}
and  $\mathcal{B}_{\pm}u_N(\pm1)=u_N(\pm1)$ for the Dirichlet BCs \eqref{IntroDBC1.4} or $\mathcal{B}_{\pm}u_N(\pm1)=\dxc^{\alpha-1}u_N(\pm1)$ for the Caputo FNBCs \eqref{IntroFNBC1.5}.

Taking $u_N(x)=\sum_{k=0}^{N}\tilde{u}_{k}L_{k}(x)$ and
letting the test functions be $L_i(x), 0\leq i\leq N-2$ gives the linear system
\begin{equation*}
  (-\widetilde{\mathcal{S}} + c\widetilde{\mathcal{M}} +\widetilde{\mathcal{B}})U=\mathcal{F},
\end{equation*}
where ${U}=(\tilde{u}_0,\tilde{u}_1,\cdots,\tilde{u}_{N})^T$, and similarly, the stiffness and mass matrix $\widetilde{\mathcal{S}}$ and $\widetilde{\mathcal{M}}$ have the same elements as the matrix $\widetilde{S}$ and $\widetilde{M}$ in \eqref{CLMS} except that the last two rows are equal to zero. The first $N-1$ rows of the matrix $\widetilde{\mathcal{B}}$ are equal to zero, the last two rows of the matrix $\widetilde{\mathcal{B}}$ are given by
$$\widetilde{\mathcal{B}}_{N-1,k} = \mathcal{B}_{-}L_{k}(-1), \; \widetilde{\mathcal{B}}_{N,k} = \mathcal{B}_{+}L_{k} (1), $$
and $\mathcal{F} = [\widetilde{\mathcal{F}}(0:N-1);g_1;g_2]$ where $\widetilde{\mathcal{F}}$ is given by \eqref{CLMS}.



\bibliographystyle{siamplain} 

\end{document}